\documentclass[11pt]{amsart}
\usepackage{amsmath,amscd,amssymb,amsfonts,amsthm}

\newtheorem{theorem}{Theorem}[section]

\newtheorem{proposition}[theorem]{Proposition}

\newtheorem{lemma}[theorem]{Lemma}
\theoremstyle{remark}
\newtheorem{remark}[theorem]{Remark}
\newtheorem{definition}[theorem]{Definition}
\newtheorem{remarks}[theorem]{Remarks}
\newtheorem{example}[theorem]{Example}

\newcommand\A{\mathcal{A}}

\renewcommand\L{\mathcal{L}}
\newcommand\G{\mathcal{G}}

\newcommand{\TM}{\mathbb{T}M}

\newcommand{\T}{\mathbb{T}}

\newcommand{\n}{\mathfrak{n}}

\newcommand{\R}{\mathbb{R}}
\newcommand{\C}{\mathbb{C}}
\newcommand{\id}{\on{id}}

\newcommand{\pr}{\on{pr}}

\newcommand\lie[1]{\mathfrak{#1}}
\renewcommand{\k}{\lie{k}}
\newcommand{\h}{\lie{h}}

\newcommand{\g}{\lie{g}}

\newcommand{\Exp}{\on{Exp}}

\renewcommand{\t}{\lie{t}}

\newcommand{\on}{\operatorname}

\newcommand{\Ad}{ \on{Ad} }
\newcommand{\ad}{ \on{ad} }

\newcommand{\Hom}{ \on{Hom}} 
\renewcommand{\ker}{ \on{ker}} 
 
\newcommand{\SU}{ \on{SU}}

\newcommand{\Mult}{{\on{Mult}}}

\newcommand{\Add}{ \on{Add}}

\renewcommand\a{\mathsf{a}}

\newcommand{\hra}{\hookrightarrow}

\newcommand{\ra}{\rightarrow}
\newcommand{\rra}{\rightrightarrows}

\renewcommand{\d}{{\mbox{d}}}
\newcommand{\dd}{\mf{d}}

\newcommand{\ol}{\overline}

\newcommand\sig{\sigma}

\newcommand\Om{\Omega}
\newcommand\om{\omega}

\newcommand{\f}{\frac}

\newcommand{\p}{\partial}
\renewcommand{\l}{\langle}
\renewcommand{\r}{\rangle}
\newcommand{\hh}{{ \f{1}{2}}}
\newcommand{\ti}{\tilde}

\newcommand\beqn{\begin{equation}}
\newcommand\eeqn{\end{equation}}
\newcommand{\ca}{\mathcal}

\newcommand{\wt}{\widetilde}
\newcommand{\iz}{\mathsf{i}}
\newcommand{\sz}{\mathsf{s}}
\newcommand{\tz}{\mathsf{t}}
\newcommand{\mf}{\mathfrak}
\newcommand{\beq}{\begin{eqnarray*}}
\newcommand{\eeq}{\end{eqnarray*}}

\newcommand{\Cour}[1]      {[\![#1]\!]}

\begin{document}

\title[]{Linearization of Poisson Lie group structures}

\author{A. Alekseev} \address{University of Geneva, Section of
  Mathematics, 2-4 rue du Li\`evre, 1211 Gen\`eve 24, Switzerland}
\email{alekseev@math.unige.ch}

\author{E. Meinrenken}
\address{University of Toronto, Department of Mathematics,
40 St George Street, Toronto, Ontario M4S2E4, Canada }
\email{mein@math.toronto.edu}

\date{\today}

\begin{abstract}
We show that for any coboundary Poisson Lie group $G$, the Poisson structure on $G^*$ is 
linearizable at the group unit.  This strengthens a result of Enriquez-Etingof-Marshall \cite{enr:co} who had 
established formal linearizability of $G^*$ for quasi-triangular Poisson Lie groups $G$.  We also prove 
linearizability properties for the group multiplication in $G^*$ and 
for Poisson Lie group morphisms, with similar assumptions.  
\end{abstract}

\maketitle
\setcounter{tocdepth}{2}

{\small \tableofcontents \pagestyle{headings}}

\setcounter{section}{-1}
\vskip.3in
\section{Introduction}\label{sec:intro}
Let $(M,\pi)$ be a Poisson manifold. If the Poisson bivector field $\pi$ vanishes at some given point $x_0\in M$, then the 
tangent fiber $T_{x_0}M$ at that point acquires a linear Poisson structure $\pi_0$. The Poisson structure on $M$ is called 
\emph{linearizable at $x_0$} if there exists a \emph{Poisson linearization}, i.e., a germ of a Poisson diffeomorphism 
$f\colon T_{x_0}M\to M$ whose differential at zero is the identity map of $T_{x_0}M$.

Any linear Poisson structure on a vector space arises as the
\emph{Lie--Poisson structure} for a Lie bracket on the dual space. For a Poisson structure vanishing at $x_0$, one calls $T_{x_0}^*M$,  
with the Lie bracket defined by $\pi_0$, 
the \emph{isotropy Lie algebra} at $x_0$.  Weinstein \cite{wei:loc}
proved that if the isotropy Lie algebra is semisimple, then 
$\pi$ is \emph{formally} linearizable at $x_0$, but need not be smoothly linearizable. Using hard estimates of Nash-Moser type, Conn \cite{con:noran} proved analytic linearizability in the semisimple case and smooth linearizability in the compact semisimple case \cite{con:norsm}.  A soft proof of Conn's linearizability theorem was
obtained more recently by Crainic and Fernandes \cite{cra:geo}. See
\cite{duf:lin,fer:lin,wad:nor,wei:pri} for further results on the
linearizability problem.

If $G$ is a Poisson Lie group, then the group unit $e$ is a 
zero of the Poisson structure. Hence, its Lie algebra $\g$ has a linear Poisson structure, corresponding to a Lie bracket on $\g^*$. 
Cahen-Gutt-Rawnsley \cite{cah:non} showed that if $G$ is a compact simply connected Lie group equipped with the Lu-Weinstein Poisson structure, then $G$ is \emph{not} linearizable at $e$ unless it is a product of $\SU(2)$'s. On the other hand, by the Ginzburg-Weinstein theorem \cite{gi:lp} the dual Poisson Lie group $G^*$ is 
linearizable -- in fact it is \emph{globally} Poisson diffeomorphic to $\g^*$. Further examples of linearizable and non-linearizable Poisson Lie group structures were obtained by Chloup-Arnould \cite{chl:lin}.

In 2005, Enriques-Etingof-Marshall  \cite{enr:co} proved that if $G$ is a factorizable quasi-triangular Poisson Lie group, then $G^*$ is formally linearizable. In this paper, we will 
strengthen this result to \emph{smooth} linearizability, for the larger class of coboundary Poisson Lie groups. We also prove 
linearizability properties of the group multiplication in $G^*$ and 
for Poisson Lie group morphisms, with similar assumptions.  

To explain our results in more detail, let $G$ be a Poisson Lie group. By Drinfeld's theory, the Lie bialgebra structure on $\g$ is equivalent to a Manin triple 
$(\dd,\g,\h)$. Here $\dd$ is a quadratic Lie algebra, with an invariant non-degenerate symmetric bilinear form 
(`metric') $\l\cdot,\cdot\r$, and  $\g$ and $\h$ are transverse Lagrangian Lie subalgebras. 
The bilinear form on $\dd$ identifies $\dd/\g$ with the dual space $\g^*$. 

Suppose for a moment that the Lie algebra Manin triple $(\dd,\g,\h)$ integrates to a triple of Lie groups $(D,G,H)$ such that 
$G,H\subset D$ are closed Lie subgroups and the product map $H\times G\to D$ 
is a global diffeomorphism. Then there is a well-defined quotient map $D\to D/G=H=G^*$. The action of $G$ by left 
multiplication on $D$ descends to the dressing action on $G^*$. 
As is well-known, the symplectic leaves of $G^*$ are the orbits of the dressing action. Hence, any linearization map should take the dressing orbits in $G^*$ to the coadjoint orbits in $\g^*$. 

In the general case, we may take $D,G,H$ to be simply connected Lie groups integrating $\dd,\g,\h$. The projection 
$D\to H=G^*$ need not be globally well-defined (e.g., the subgroup of $D$ integrating $\h \subset \dd$
need not be simply connected), but it is defined on a neighborhood of the group unit (hence as a germ of a map), since 
the product map $H\times G\to D,\ (u,g)\mapsto ug$ is  
a diffeomorphism near the group units. 

As noted by Drinfeld (see Section \ref{sec:cob}), a coboundary structure on the Lie bialgebra $\g$ is equivalent to a $\g$-equivariant splitting of the sequence 
\[ 0\to \g\to \dd\to \g^*\to 0.\] 
We stress that the image of the splitting $\g^*\to \dd$ is usually different from $\h$ (which need not be $\g$-invariant). Define 
\[ \Exp\colon \g^*\to G^*\]
by the composition $\g^*\to \dd\to D\to G^*$, where $\exp: \dd \to D$ is the exponential map. Strictly speaking, since the projection $D\to G^*$ is only well-defined on some neighborhood of the group unit, the map $\Exp$ is only defined near $0\in\g^*$, or as a germ of a smooth map. We think of 
$\Exp$ as a `modified exponential map'. In contrast to the usual exponential map $\exp: \g^* \to G^*$, it intertwines the coadjoint action 
with the dressing action, hence it takes symplectic leaves to 
symplectic leaves. 

A smooth map $\phi\colon \g^*\to G$ will be called a \emph{bisection} if $\A(\phi)\colon \g^*\to \g^*,\ \mu\mapsto \phi(\mu).\mu$ is a diffeomorphism of $\g^*$.  The name 
comes from an interpretation of $\phi$ as a bisection of the symplectic groupoid $T^*G\rra \g^*$.  Our first result is the following theorem:
\begin{theorem}
Let $\g$ be a coboundary Lie bialgebra, and $G^*$ the Poisson Lie
group with Lie bialgebra $\g^*$.  Then $G^*$ is smoothly linearizable
at $e$. In fact there exists a germ of a bisection $\phi\colon \g^*\to G$, with the property that 
\[ \Exp\circ \A(\phi)^{-1}\]
is a Poisson linearization of $G^*$ at $e$.
\end{theorem}
Our next result is a functoriality property for the Poisson linearizations. It generalizes a similar result for Ginzburg-Weinstein maps 
obtained in \cite{al:gw}.

\begin{theorem}
Let  $\g_1\to \g_2$ be a morphism of Lie bialgebras, with dual morphism
$\tau\colon \g_2^*\to \g_1^*$ exponentiating to the Poisson Lie group morphism $\ca{T}\colon G_2^*\to G_1^*$. Suppose that $\g_1$ and $\g_2$ have coboundary structures, 
determining equivariant maps $\Exp_i\colon \g_i^*\to G_i^*,\ i=1,2$ 
as above.  If $\phi_1\colon \g_1^*\to G_1$ is a germ of a bisection 
such that $\on{Exp}_1\circ \A(\phi_1)^{-1}$ is Poisson, then it is possible to choose a bisection $\phi_2\colon \g_2^*\to G_2$ such that $\on{Exp}_2\circ \A(\phi_2)^{-1}$ is Poisson, and such that furthermore
\[\ca{T}\circ (\on{Exp}_2\circ \A(\phi_2)^{-1})=
(\on{Exp}_1\circ \A(\phi_1)^{-1})\circ \tau.\]
\end{theorem}

In \cite{al:gw}, we used the functoriality property to show that in the case of $G={\rm U}(n)$, the Gelfand-Zeitlin completely
integrable system on $\g^*={\rm u}(n)^*$ defined by Guillemin-Sternberg \cite{gu:gc} is isomorphic to the completely integrable
system defined on $G^*={\rm U}(n)^*$ by Flaschka-Ratiu \cite{fl:mo}.

By definition of a Poisson Lie group, the multiplication in $G^*$ is a Poisson morphism
\[ \on{Mult}_{G^*}\colon G^*\times G^*\to G^*\]
as is the addition map $\on{Add}_{\g^*}\colon \g^*\times\g^*\to \g^*$. We have
\begin{theorem}
Let $\g$ be a coboundary Lie bialgebra, and $G^*$ the Poisson Lie group with Lie bialgebra $\g^*$.  Let $\psi\colon \g^*\to G$ be a germ of a bisection for which $\Exp\circ \A(\psi)^{-1}$ is 
Poisson. Then it is possible to choose a germ of a bisection 
$\phi\colon \g^*\times\g^*\to G\times G$ such that the 
map $(\Exp\times\Exp)\circ \A(\phi)^{-1}$ is Poisson, and furthermore
\[ \Mult_{G^*}\circ ((\Exp\times\Exp)\circ \A(\phi)^{-1})
=\Exp\circ (\A(\psi)^{-1}\circ \on{Add}_{\g^*}).\]
\end{theorem}

It is thus possible to choose the Poisson linearizations of $G^*$ and of $G^*\times G^*$  in such a way that the multiplication in $G^*$ 
becomes the addition in $\g^*$. In the case of a compact Lie group $G$ with Lu-Weinstein Poisson structure, the result holds \emph{globally}, not only on the level of germs. It has the following interesting corollary: let
$\mathcal{O}_1, \mathcal{O}_2 \subset \g^*$ be coadjoint orbits 
and  $\mathcal{D}_1=\Exp(\mathcal{O}_1), \mathcal{D}_2=\Exp(\mathcal{O}_2)$ be the corresponding dressing orbits. Then,
$$
\Exp(\mathcal{O}_1+\mathcal{O}_2) = \mathcal{D}_1\mathcal{D}_2.
$$

\vskip.4in
\noindent{\bf Acknowledgments.} It is a pleasure to thank David Li-Bland and Jiang-Hua Lu for discussions and helpful comments. 
A.A. was supported by the grant MODFLAT of the European Research Council and by the grants
140985 and  141329 of the Swiss National Science Foundation.
E.M. was supported by an NSERC Discovery Grant.

\vskip.4in
\noindent{\bf Conventions.} 
The \emph{flow} $F_t$ of a time dependent vector field $v_t$ on a manifold $M$ is defined in terms of the action on functions 
by the equation 
\[ \L(v_t)\circ (F_t^{-1})^*=\f{\p}{\p t} (F_t^{-1})^*;\]
%
%\[ \L(v_t)=-(F_t^{-1})^*\circ \f{\p}{\p t} (F_t^*).\] 
%
equivalently 
\begin{equation}\label{eq:timeder}
\f{\p}{\p t}-\L(v_t)=(F_t^{-1})^*\circ \f{\p}{\p t} \circ F_t^*
\end{equation}
as operators on $t$-dependent functions. 
The same formula applies to the action on $t$-dependent differential forms or other tensor fields.

Given a bivector field $\pi$ on a manifold $M$, we denote by $\pi^\sharp\colon T^*M\to TM$ the bundle map
$\mu\mapsto \iota(\mu)\pi=\pi(\mu,\cdot)$. For a 2-form $\om$ we denote by $\om^\flat$ the 
bundle map $TM\to T^*M,\ v\mapsto \om(v,\cdot)$. 

An action of a Lie algebra $\g$ on a manifold $M$ is a Lie algebra morphism $\g\to \Gamma(TM),\ \xi\mapsto \xi_M$ to the Lie algebra of vector fields such that the action map $M\times\g\to TM$ is smooth. An action of a Lie group $G$ on $M$ is a Lie group morphism 
$\A\colon G\to \on{Diff}(M)$ such that the action map 
$G\times M\to M,\ (g,m)\mapsto g.m:=\A(g)(m)$ is smooth. 
A Lie group action defines an action of its Lie algebra $\g$ by $\xi_M(f)=\f{\p}{\p t}|_{t=0} f(\exp(-t\xi).m)$. The coadjoint action of $G$ on $\g^*$ is denoted $\Ad^*$; thus $\Ad^*(g)=(\Ad(g^{-1}))^*$. 
Similarly we write $\ad^*(\xi)=-\ad(\xi)^*$ for the coadjoint representation of $\g$ on $\g^*$.

\section{Moser method for Poisson manifolds}\label{sec:moser}
It is a well-known fact that Lie algebra structures on a vector space $\g$ are in 1-1 correspondence with linear Poisson structures on 
the dual space $\g^*$. If $e_a$ is a basis of $\g$, defining structure constants $f_{ab}^c$ with $[e_a,e_b]=\sum_c f_{ab}^c e_c$, 
then the corresponding \emph{Lie--Poisson structure} $\pi_{\g^*}$ on $\g^*$ reads as
\[ \pi_{\g^*}=-\hh \sum_{abc} f_{ab}^c \mu^c  \f{\p}{\p \mu_a}\wedge \f{\p}{\p \mu_b}.\]
The vector fields $\xi_{\g^*}=-\pi_{\g^*}^\sharp \l\d\mu,\xi\r,\ \ \xi\in\g$, are the generators for the coadjoint action. In this paper, we will study Poisson Lie groups $(G,\pi_G)$ whose dual 
$(G^*,\pi_{G^*})$ is linearizable, i.e. Poisson diffeomorphic to 
$(\g^*,\pi_{\g^*})$ near the group unit.  
Our main technique in proving linearizability are \emph{gauge transformations} \cite{bu:ga}, together with the \emph{Moser method} for Poisson manifolds \cite{al:ka,al:gw,al:li}. We will find it convenient to develop this theory within the frameworks of Dirac geometry 
\cite{al:pur,bu:ga,cou:di} and symplectic groupoids \cite{cos:gro,wei:sym,kar:ana}.

\subsection{Dirac geometry}
For any manifold $M$, let $\TM=TM\oplus T^*M$ with the metric 
\[ \l(v_1,\alpha_1),(v_2,\alpha_2)\r
=\alpha_1(v_2)+\alpha_2(v_1).\] 
The space of sections $\Gamma(\TM)$ comes equipped with the 
\emph{Courant bracket} 
\begin{equation}
\label{eq:courbracket}\Cour{(v_1,\alpha_1),(v_2,\alpha_2)}=([v_1,v_2],\L_{v_1}\alpha_2-\iota_{v_2}\d\alpha_1)
\end{equation}
for vector fields $v_1,v_2$ and 1-forms $\alpha_1,\alpha_2$.
A maximal isotropic subbundle $E\subset \TM$ whose space of sections is closed under this bracket is called a 
\emph{Dirac structure} on $M$. A Poisson structure $\pi$ defines a Dirac structure given by the graph of $\pi$
\[\on{Gr}(\pi)=\{(\pi^\sharp(\alpha),\alpha)|\ \alpha\in T^*M\};\]
conversely, any Dirac structure with $E\cap TM=0$ defines a Poisson structure in this way.

Suppose $F\colon M\to M'$ is a smooth map and $\sigma\in\Om^2(M)$ a closed 2-form. Given $(v,\alpha)\in \TM=TM\oplus T^*M$ and $(v',\alpha')\in 
\TM'=TM'\oplus T^*M'$, we write 
\begin{equation}\label{eq:related}
(v,\alpha)\sim_{(F,\sigma)} (v',\alpha')
\ \Leftrightarrow \ v'=TF(v) \mbox{ and }(TF)^*\alpha'=\alpha+\iota(v)\sigma.
\end{equation}
Given Dirac structures $E\subset \TM$ and $E'\subset \TM'$, we say that $(F,\sigma)$ defines a (strong) Dirac 
morphism if every element of $E_{F(m)}'$ is $(F,\sigma)$-related to a unique element of $E_m$. In the case of 
Poisson structures $E=\on{Gr}(\pi),\ E'=\on{Gr}(\pi')$ we will call $(F,\sigma)$ a \emph{twisted Poisson map}. (It is an ordinary  
Poisson map for $\sigma=0$.) 

The composition of  two \emph{morphisms} 
$(F,\sigma)\colon M\to M'$ and $(F',\sigma')\colon M'\to M''$ is 
given as 
\[ (F',\sigma')\circ (F,\sigma)=(F'\circ F,\sigma+F^*\sigma').\]
In particular, if $F$ is invertible, then so is $(F,\sigma)$ 
with inverse $(F^{-1},-(F^{-1})^*\sigma)$. 

\subsection{Gauge transformations}
Similar to \eqref{eq:related} one has a notion of relation between \emph{sections} of 
$\TM$ and $\TM'$. The relation of sections is compatible with the metric and with the Courant bracket. Hence, the 
semi-direct product $\on{Diff}(M)\ltimes \Om^2_{\on{closed}}(M)$ acts by automorphisms of $\TM$, preserving the 
Courant bracket and the metric. In particular, it takes Dirac structures to Dirac structures.

Let $(M,\pi)$ be a Poisson manifold.  A closed 2-form $\sigma$ on $M$
is said to define a \emph{gauge transformation} \cite{bu:ga} of $\pi$ if 
the image of $\on{Gr}(\pi)$ 
under the morphism $(\on{id}_M,\sigma)$ is transverse to $TM$, hence is again a Poisson structure
$\pi^\sigma$. The transversality condition is equivalent to invertibility of the 
bundle
map $I+\sigma^\flat\circ \pi^\sharp$, and one has
\begin{equation}\label{eq:bunmaps}
 (\pi^\sigma)^\sharp=\pi^\sharp \circ (I+\sigma^\flat\circ \pi^\sharp)^{-1}.
\end{equation}
This Poisson structure $\pi^\sigma$ has the same symplectic leaves as $\pi$, but with the symplectic form on the leaves 
changed by the  pull-back of $\sigma$. Note also that a morphism $(F,\sigma)$ between two Poisson manifolds $M,M'$
is a $\sigma$-twisted Poisson map if and only if $\sigma$ defines 
a gauge transformation of $\pi$, and $F$ is an ordinary Poisson map 
$F\colon (M,\pi^\sigma)\to (M',\pi')$.

\subsection{Hamiltonian actions on Poisson manifolds}\label{subsec:ham}
Let $\g^*$ be the dual of a Lie algebra, equipped with the Lie-Poisson structure.  As is well-known, 
the cotangent Lie algebroid $T^*\g^*$ is canonically isomorphic to the action Lie algebroid for the coadjoint action. 
The isomorphism is given by the map $e_{\g^*}\colon \g\to \Gamma(\on{Gr}(\pi_{\g^*}))$, where 
\[ e_{\g^*}(\xi)=(\xi_{\g^*},-\l\d\mu,\xi\r)\]
for $\xi\in\g$.  If $\Phi\colon M\to \g^*$ is a Poisson map, then there are uniquely defined sections $e_M(\xi)$ of 
$\on{Gr}(\pi)$ such that 
\[ e_M(\xi)\sim_{(\Phi,0)} e_{\g^*}(\xi).\] 
Write $e_M(\xi)=(\xi_M,-\l\d\Phi,\xi\r)$. Then $\xi\mapsto \xi_M$ defines a Hamiltonian $\g$-action on $M$, with $\Phi$ as its moment map.  That is, $\Phi$ is $\g$-equivariant and 
\[ \xi_M=-\pi_M^\sharp \l\d\Phi,\xi\r \]
for all $\xi\in\g$.  
%Suppose $\sigma$ is $\g$-invariant, and that $\Psi\colon M\to \g^*$ is a moment map 
%for $\sigma$ so that $\iota(\xi_M)\sigma=-\d\l\Psi,\xi\r$. 
%for all $\xi\in\g$. 
%Then the $\g$-action is also Hamiltonian with respect to $\pi_M^\sigma$, with the new moment map $\Phi^\sigma=\Phi+\Psi$.
%
For example, the addition map $\on{Add}\colon \g^*\times\g^*\to \g^*$ 
satisfies $e_{\g^*\times\g^*}(\xi,\xi)\sim_{\on{Add}}e_{\g^*}(\xi)$, hence it is the moment map for the diagonal $\g$-action on 
$\g^*\times \g^*$. 

Given Hamiltonian $\g$-actions on $M$ and $M'$, with moment maps 
$\Phi\colon M\to \g^*$ and $\Phi'\colon M'\to \g^*$, we say that 
$(F,\sigma)$ is a \emph{twisted morphism of 
Hamiltonian $\g$-spaces} if 
\[ e_M(\xi)\sim_{(F,\sigma)} e_{M'}(\xi)\]
for all $\xi\in\g$. This is equivalent to the $\g$-equivariance of $\Phi$ together 
with the following moment map property of $\sigma$, 
$\iota(\xi_M)\sigma=-\d\l F^*\Phi'-\Phi,\xi\r$. (In particular, $\sigma$ must be invariant.)
%
%\begin{remark}More conceptually, we say that $e_M\colon \g\to \Gamma(\on{Gr}(\pi))$ defines generators for a Poisson $\g$-action if it defines a morphism of Lie algebroids $M\times\g\to \on{Gr}(\pi)\cong T^*M$ whose composition with the anchor map is the infinitesimal action map. \end{remark}
  
%\item As an immediate consequence of \eqref{eq:symchange}, $(\pi^{\sigma_1})^{\sigma_2}=\pi^{\sigma_1+\sigma_2}$ (if the left hand side is defined).
%

%Finaly, we remark that if $f\colon (M_1,\pi_1)\to (M_2,\pi_2)$ is a Poisson map, and $\sig\in\Om^2(M_2)$ defines a gauge transformation of $\pi_2$, then $f^*\sigma_1$ is a gauge transformation of $\pi_1$, and $f$ is also a Poisson map for the gauge transformed Poisson structures, $f\colon (M_1,\pi_1^{f^*\sig_1})\to (M_2,\pi_2^\sig)$.

\subsection{Moser flows}\label{subsec:moser}
Let $\sig_t$ be a smooth family of closed 2-forms on the Poisson manifold $(M,\pi)$, 
with $\sig_0=0$, defining gauge transformations $\pi_t=\pi^{\sig_t}$ for all $t$. Suppose furthermore that 
\begin{equation}\label{eq:sigmaeq}
 \f{\p}{\p t} \sig_t=-\d a_t
\end{equation}
for a smooth family of 1-forms $a_t$. We call $a_t$ the \emph{Moser 
1-form}, and $v_t=-\pi_t^\sharp(a_t)$ the corresponding 
\emph{Moser vector field}. 
The flow $F_t$ of the Moser vector field
(called the Moser flow) intertwines the Poisson structures
(see, e.g., \cite{al:gw}):
\begin{equation}\label{eq:intert}
 \pi_t=(F_t)_*\pi.\end{equation}
Suppose that we are also given a Hamiltonian $\g$-action on $M$, with moment map $\Phi\colon M\to \g^*$. 
Let $e_M(\xi)\in\Gamma(\on{Gr}(\pi))$ be the sections defined by the condition 
$e_M(\xi)\sim_{(\Phi,0)}e_{\g^*}(\xi),\ \xi\in\g$. Suppose $\Phi_t\colon M\to \g^*$ is a family 
of maps with $\Phi_0=\Phi$ and 
\begin{equation}\label{eq:deformedmomentmap}
 e_M(\xi)\sim_{(\Phi_t,\sig_t)} e_{\g^*}(\xi),\end{equation}
for all $t$. That is, $\Phi_t$ is a moment map for the same action, relative to the gauge transformed Poisson structure 
$\pi_t$. If the family of 1-forms $a_t$ satisfies the two conditions
\begin{equation}\label{eq:ateqn}
\f{\p}{\p t}\sig_t=-\d a_t,\ \ \f{\p}{\p t}\l\Phi_t,\xi\r=\iota(\xi_M) a_t
\end{equation} 
(which in particular implies the invariance of $a_t$), then the Moser flow of $v_t$ intertwines not only the bivector fields but also the moment maps:
\[ \Phi_t=(F_t^{-1})^*\Phi.\]
\begin{remark}
In general, the time dependent vector field $v_t$ need not be complete, hence its flow may not exist for all $t$. Global existence of the flow 
of $v_t$ is guaranteed if the symplectic leaves of $M$ are compact (since $v_t$ is tangent to the leaves). 
\end{remark}

%preserving the Poisson structures $\pi_t$ and admitting a family of moment maps $\Phi_t$ Hamiltonian $\g$-action on $M$, withmoment map $\Phi\colon M\to \g^*$, we assume that the 1-forms $a_t$ are $\g$-invariant. Then so are the 2-forms $\sig_t$ and the Poisson structures $\pi_t$. Let $\Phi_t\colon M\to \g^*$ be the functions obtained by integrating $\f{\p}{\p t}\l\Phi_t,\xi\r=-\iota(\xi_M)a_t$ with initial condition $\Phi_0=\Phi$. Then $\Phi_t$ is moment map with respect to $\pi_t$, and the Moser flow satisfies $\Phi_t=(F_t^{-1})^*\Phi$. Equation \eqref{eq:intert} says that the pair $(F_t^{-1},\sigma_t)\colon M\to M$ is a family of twisted Poisson automorphisms of $(M,\pi)$. In the equivariant setting,  it is a family of twisted automorphisms of Hamiltonian $\g$-spaces.

\subsection{Symplectic groupoids}\label{subsec:bis}
In this section we interpret the Moser method for Poisson
manifolds in terms of symplectic groupoids. In this context, the Moser flow
will be given by an action of bisections. Let us first recall some definitions.

Let ${\mathcal{G}}\rra M$ be a Lie groupoid, with source and target map $\sz,\tz\colon {\mathcal{G}}\to M$ and unit map $\iz\colon M\to {\mathcal{G}}$. We denote by $\G^{(2)}\subset \G\times\G$ the submanifold of composable
elements, by 
$\Mult_{\mathcal{G}}\colon {\mathcal{G}}^{(2)}\to {\mathcal{G}}, (\gamma_1,\gamma_2)\mapsto \gamma_1 \gamma_2$ the multiplication, and by $\on{Inv}_{\mathcal{G}}\colon {\mathcal{G}}\to {\mathcal{G}}$ the inversion. An embedded submanifold of $\G$ on which both 
$\sz$ and $\tz$ restrict to diffeomorphisms is called a 
\emph{bisection} of $\G$. The set $\Gamma(\G)$ of bisections is a group under groupoid multiplication \footnote{If $\G$ is a bundle of Lie groups, so that the source and target map coincide, then a bisections is just an ordinary section.}, with group unit the identity bisection
$M\subset \G$. The Lie algebroid $A\G$ of $\G$ is defined as follows: As a vector bundle, 
$A\G=\nu(M,\G)$ is the normal bundle of $M$ inside $\G$; the Lie bracket is induced from the group commutator 
on  $\Gamma(\G)$, and the anchor $\a\colon A\G\to TM$ is induced from the difference $T\tz-T\sz\colon T\G\to TM$. 

From now on we will regard bisections as 
sections $\phi\colon M\to {\mathcal{G}}$ of the source map
(i.e., $\sz\circ \phi=\on{id}_M$), with the additional property that
\begin{equation}\label{eq:action}\A(\phi):=\tz\circ \phi\end{equation}
is a diffeomorphism of $M$. Note that $\A\colon \Gamma(\ca{G})\to \on{Diff}(M)$ is a group action. 
%(It corresponds to the groupoid action $\ca{G}\times_M M\to M,\ (\gamma,s(\gamma))\mapsto t(\gamma)$.) 
In addition, there are two commuting actions  $\A^L,\ \A^R\colon \Gamma(\ca{G})\to \on{Diff}(\ca{G})$  
on the groupoid $\mathcal{G}$: 
\[ \A^L(\phi)(\gamma)=\phi(\tz(\gamma))\, \gamma,\ \ \ \ 
\A^R(\phi)(\gamma)=\gamma\,\phi(\sz(\gamma))^{-1}.\]
%
%The target map intertwines $\A^L$ with $\A$ and $\A^R$ with $\id_M$; 
%the source map intertwines $\A^L$ with $\id_M$ and $\A^R$ with $\A$. 
These satisfy 
\[ \tz \circ \A^L(\phi)=\A(\phi)\circ \tz,\ \
\tz \circ \A^R(\phi)=\tz,\ \ \sz \circ \A^L(\phi)=\sz,\ \ \sz \circ \A^R(\phi)=\A(\phi)\circ \sz.\]
The groupoid inversion interchanges the 
actions $\A^L$ and $\A^R$. 
%and 
%\[ \A^L(\phi)\circ \on{Inv}_\G=\on{Inv}_\G \circ \A^R(\phi), \ \ 
%\A^R(\phi)\circ \on{Inv}_\G=\on{Inv}_\G \circ \A^L(\phi).\]
%
We will also consider the adjoint action 
\[ \Ad(\phi)(\gamma)=\phi(\tz(\gamma))\gamma\phi(\sz(\gamma))^{-1};\] 
this action is by automorphisms of the Lie groupoid $\G$, with underlying map $\A(\phi)$. 
It induces an action by Lie algebroid automorphisms of $A\G$, which is again denoted by $\phi\mapsto \Ad(\phi)$.

A differential form $\om\in \Om(\G)$ on the groupoid is called \emph{multiplicative} if it has the property
\begin{equation}\label{eq:symplgr}
 \Mult_{\mathcal{G}}^*\om=\pr_1^*\om+\pr_2^*\om
\end{equation}
where $\pr_i\colon {\mathcal{G}}^{(2)}\to {\mathcal{G}}$ are the two 
projections. 
A groupoid with a multiplicative symplectic 2-form $\om\in\Om^2(\G)$ is called a \emph{symplectic groupoid} \cite{cos:gro,wei:sym,kar:ana}.
For any symplectic groupoid, the inclusion $\iz\colon M\to \ca{G}$ is 
a Lagrangian embedding, and  
%$\on{Inv}_{\mathcal{G}}^*\om=-\om$. 
%
%Furthermore, 
the tangent spaces to  the $\tz$-fibers and to the $\sz$-fibers are $\om$-orthogonal:  
\[ \ker(T\tz)^\om=\ker(T\sz).\]
The manifold $M$ inherits a Poisson structure $\pi=\pi_M$ for which the target map $\tz$ is  Poisson and the 
source map $\sz$ is anti-Poisson. That is, 
\[ \pi_\G\sim_\tz \pi,\ \ \pi_\G\sim_\sz -\pi,\]
where $\pi_\G$ is the Poisson structure on $\G$ given by 
$\pi_\G^\sharp\circ \om^\flat=\id$. One calls the symplectic groupoid $({\mathcal{G}},\om)$ an \emph{integration} of the Poisson manifold $(M,\pi)$. Since $M\subset \G$ is Lagrangian, the symplectic form on $\G$ gives a
non-degenerate pairing between
$TM$ and $\nu(M,\G)=A\G$, thus identifying $A\G\cong T^*M$ in such a way
that the following diagram commutes:
\begin{equation}\label{eq:identification}
\begin{CD}
{T\G|_M} @>{\om^\flat}>> {T^*\G|_M}\\
@VVV @VVV\\
A\G @>>{\cong}> {T^*M}
\end{CD} \end{equation}
The anchor map for the resulting cotangent Lie algebroid
is the map $\pi^\sharp\colon T^*M\to TM$.

%Since $M\subset \G$ is Lagrangian, the symplectic form on $\G$ gives a non-degenerate pairing between 
%$TM$ and $\nu(M,\G)=A\G$, thus identifying $A\G=T^*M$. The anchor map for the resulting cotangent Lie algebroid 
%is the map $\pi^\sharp\colon T^*M\to TM$.

\begin{proposition}\label{prop:integration}
Suppose $({\mathcal{G}},\om)$ integrates $(M,\pi)$, and  $\sigma\in \Om^2(M)$ is a closed 2-form defining a gauge 
transformation of $\pi$. Then 
\[ \om+\tz^*\sigma,\ \ \om-\sz^*\sigma,\ \ \om+\tz^*\sigma-\sz^*\sigma\]
are all symplectic. Furthermore, 
$(\G,\om+\tz^*\sigma-\sz^*\sigma)$ integrates $(M,\pi^\sigma)$. 
\end{proposition}
\begin{proof}
Since $\sig$ defines a gauge transformation of $\pi$, and $\tz$ is a Poisson map, the pull-back $\tz^*\sigma$ defines a gauge transformation of $\pi_\G$. Similarly, $-\sz^*\sigma$ defines a gauge transformation. By definition, the gauge transformed Poisson structure on $\G$ is again non-degenerate. From $\pi_\G(\sz^*\alpha,\tz^*\beta)=0$ one obtains that $\pi_\G^\sharp\circ \tz^*=
(\pi_\G^{-\sz_*\sigma})^\sharp\circ \tz^*$, hence
$\tz_*\pi_\G^{-\sz^*\sigma}=\pi$. We may hence iterate the gauge transformations, obtaining a well-defined and non-degenerate 
Poisson structure 
$(\tz_*\pi_\G^{-\sz^*\sigma})^{\tz^*\sigma}=\pi_\G^{\tz^*\sigma-\sz^*\sigma}$, with corresponding symplectic form
$\om+\tz^*\sigma-\sz^*\sigma$. Since the latter is multiplicative, the proof is complete.
\end{proof}

\begin{remark}
While Poisson manifolds need not admit an integration 
to a symplectic groupoid, there always exists an integration to a \emph{local} symplectic groupoid, unique up to isomorphism. 
Local Lie groupoids $\G\rra M$ are defined similar to local Lie groups \cite{pra:tro}: Roughly, instead of a globally defined multiplication
$\Mult_\G\colon \ca{G}^{(2)}\to \G$, the map $\Mult_\G$ is only defined on an open neighborhood of $(M\times\G)\cup (\G\times M)\subseteq \ca{G}^{(2)}$, and the axioms are only required for elements on which the products are defined.  
 A nice geometric construction 
of the local symplectic groupoid integrating a Poisson manifold was obtained by Crainic and Marcut \cite{cra:exi}. 
The considerations in this section apply to the local symplectic groupoids with simple modifications. 
\end{remark}

%\[ \begin{split}\sz\circ \A^L(\phi)=\sz,&\ \ \tz\circ \A^L(\phi)=\A(\phi)\circ \tz,\\ \tz\circ \A^R(\phi)=\tz,&\ \ \sz\circ \A^R(\phi)=\A(\phi)\circ \sz.\end{split}\]

\subsection{Action of bisections on cotangent algebroid}
\begin{proposition}\label{prop:bisar}
For any symplectic groupoid $(\G,\om)$ over $(M,\pi)$, the group $\Gamma(\G)$ of bisections acts on $(M,\pi)$ by twisted Poisson automorphisms, via the group homomorphism
\begin{equation}\label{eq:dagr}
 \Gamma(\G)\to \on{Diff}(M)\ltimes \Om^2(M),\ \phi\mapsto (\A(\phi),\sigma_\phi)\end{equation}
where 
\begin{equation}\label{eq:sigphi} 
\sigma_\phi=\phi^*\om\end{equation}
Similarly, each of the following maps defines actions of $\Gamma(\G)$ on $\G$ by twisted Poisson automorphisms:
\begin{equation}\label{eq:bisectionaction}
 \phi\mapsto (\A^L(\phi),\,\tz^*\sig_\phi),\ \ \ \phi\mapsto(\A^R(\phi),\,-\sz^*\sig_\phi),\ \ \ 
\phi\mapsto (\Ad(\phi),\,\tz^*\sig_\phi-\sz^*\sig_\phi).\end{equation}
\end{proposition}
\begin{proof}
The claim that \eqref{eq:bisectionaction} are twisted Poisson automorphisms amounts to the following identities:
\begin{equation}\label{eq:transform}
 \A^L(\phi)^*\om=\om+\tz^*\sigma_\phi,\ \ \A^R(\phi)^*\om=\om-\sz^*\sigma_\phi,\ \ \Ad(\phi)^*\om=\om+\tz^*\sigma_\phi-\sz^*\sigma.\end{equation}
Since $\A^L(\phi)=\Mult_\G\circ (\phi\circ \tz,\ \on{id}_{\mathcal{G}})$, the multiplicativity property  \eqref{eq:symplgr} shows that
$\A^L(\phi)^*\om=\om+\tz^*\circ \phi^*\om
=\om+\tz^*\sigma_\phi$ proving the first identity in \eqref{eq:transform}. The second is obtained similarly, and 
the third follows by iteration
\[ \Ad(\phi)^*\om=\A^R(\phi)^*\A^L(\phi)^*\om=\A^R(\phi)^*(\om+\tz^*\sigma)=\om+\tz^*\sigma-\sz^*\sigma\]
and since $\tz\circ \A^R(\phi)=\tz$. 
Applying the map $\A(\phi^{-1})_*\circ \tz_*=\tz_*\circ \A^L(\phi^{-1})_*$ to $\pi_\G$ one obtains
\[
\A(\phi^{-1})_* \pi = \A(\phi^{-1})_*\circ \tz_* \pi_\G = \tz_*\circ \A^L(\phi^{-1})_* \pi_\G = \tz_* \pi_\G^{\tz^*\sig_\phi} =
\pi^{\sig_\phi}
\]
proving that $(\A(\phi),\sigma_\phi)$ is a twisted Poisson automorphism. 
%
%\[ \A^L(\phi^{-1})_*\pi_\G=\pi_\G^{\tz^*\sig_\phi},\]
%one obtains $\pi^{\sig_\phi}=\A(\phi^{-1})_*\pi$, proving that $(\A(\phi),\sigma_\phi)$ is a twisted Poisson automorphism. 
\end{proof}

The Proposition shows that the action of $\Gamma(\G)$ on $\T M$, given by $\phi\mapsto (\A(\phi),\sigma_\phi)$, preserves 
the sub-bundle $\on{Gr}(\pi)\subset \T M$. It hence defines an action by automorphisms of the cotangent Lie algebroid 
$T^*M\cong \on{Gr}(\pi)$ (where the identification is given by $\alpha\mapsto (\pi^\sharp(\alpha),\alpha)$). On the other hand, 
we have the adjoint action of $\Gamma(\G)$ on the Lie algebroid $A\G$. The two actions 
agree:
\begin{proposition}\label{prop:agcot}
The  isomorphism $A\G\cong T^*M$ given by $\om$ intertwines the adjoint 
action of $\Gamma(\G)$ on $A\G$ with the action on the cotangent Lie algebroid.
\end{proposition}
\begin{proof}
Let us first derive a formula for the natural action on $T^*M$.
The image $\alpha'$ of an element $\alpha\in T^*M$ under the action of $\phi\in\Gamma(\G)$ is determined 
by $(\pi^\sharp(\alpha),\alpha)\sim_{(\A(\phi),\sigma_\phi)} (\pi^\sharp(\alpha'),\alpha')$. Hence
\[ \A(\phi)^*\alpha'=\alpha+\iota(\pi^\sharp(\alpha))\sigma_\phi=(1+\sigma_\phi^\flat\circ \pi^\sharp)(\alpha).\]
This shows that the action of $\Gamma(\G)$ on $T^*M$ is given by 
\[ \phi\mapsto (\A(\phi)^{-1})^*\circ (1+\sigma_\phi^\flat\circ \pi^\sharp).\]

Recall next that the isomorphism $\om^\flat: T\G \to T^*\G$  descends to an isomorphism 
\[ A\G=(T\G|_M)/TM \to (T^*\G|_M)/\on{ann}(TM) \cong T^*M. \]
The adjoint action of $\phi$ 
on $A\G$ is induced by $T\Ad(\phi)$ on $T\G$. We first transfer this action to $T^*\G$ using
$\om^\flat$:
\[\begin{split} \om^\flat\circ T\Ad(\phi)\circ \pi_\G^\sharp&=(\Ad(\phi)^{-1})^*\circ (\Ad(\phi)^*\om)^\flat\circ \pi_\G^\sharp\\
&=(\Ad(\phi)^{-1})^*\circ (\om+\tz^*\sigma-\sz^*\sigma)^\flat\circ \pi_\G^\sharp\\
&=(\Ad(\phi)^{-1})^*\circ \big(1+(\tz^*\sigma-\sz^*\sigma)^\flat\circ \pi_\G^\sharp\big).
\end{split}\]
To compute the induced action on $T^*M$, we compose the equation above with the 
projection $\iz^*\colon T^*\G\to T^*M$ (the kernel of this projection is exactly $\on{ann}(TM)$)
from the left and any right inverse to $\iz^*$ from the right. 
It is convenient to take the latter to be $\tz^*$.
Using $T\tz\circ \pi_\G^\sharp\circ \tz^*=\pi^\sharp$ and $T\sz\circ \pi_\G^\sharp\circ \tz^*=0$ we find
\[ \begin{split}
\iz^*\circ \om^\flat\circ T\Ad(\phi)\circ \pi_\G^\sharp\circ \tz^*&=
(\A(\phi)^{-1})^*\circ \iz^*\circ \big(1+(\tz^*\sigma_\phi-\sz^*\sigma_\phi)^\flat\circ \pi_\G^\sharp\big)\circ \tz^*\\
&=(\A(\phi)^{-1})^*\circ \big(1+ \sigma_\phi^\flat\circ (T\tz-T\sz)\circ \pi_\G^\sharp\circ \tz^*\big)\\
&=(\A(\phi)^{-1})^*\circ \big(1+ \sigma_\phi^\flat\circ\pi^\sharp\big)\end{split}\]
as desired.
\end{proof}
\begin{example}
If $K$ is a Lie group, then the cotangent bundle $T^*K$ is a symplectic groupoid integrating $\k^*$ with the Lie--Poisson structure. Using left trivialization $T^*K\cong K\times\k^*$, 
the groupoid structure is that of an action groupoid for the coadjoint action. In this case, bisections may be viewed as
maps 
$\phi\colon \k^*\to K$. (We stress that in this context, $\Gamma(T^*K)$ refers to the bisections as a 
groupoid $T^*K\rra \k^*$, rather than the sections as a bundle $T^*K\ra K$.) The symplectic form on $T^*K$ reads as 
$\om=\d \l\mu,\theta^L\r$, where $\theta^L\in\Om^1(K,\k)$ is the left-invariant Maurer-Cartan 
form, and $\mu\in \Om^0(\k^*,\k^*)$ is the identity map of $\k^*$. Hence, 
\[ \sigma_\phi=\d \l \mu,\phi^*\theta^L\r.\]
\end{example}

\begin{remark}
A bisection $\phi$ of a symplectic groupoid 
is called \emph{Lagrangian} if the corresponding submanifold $\phi(M)\subseteq\G$ is a Lagrangian submanifold, or equivalently 
$\sig_\phi=0$. As an immediate consequence of the Proposition, the Lagrangian bisections form a subgroup, and the maps $\A(\phi)$ are Poisson.  More generally, $\A(\phi)$ is a Poisson automorphism of $(M,\pi)$ if the lifted map $\Ad(\phi)$ is an 
automorphism of the symplectic groupoid $(\G,\om)$, that is, 
if $\tz^*\sig_\phi=\sz^*\sig_\phi$.  Bisections with this property form a subgroup, which is usually larger than the group of Lagrangian bisections. (Consider for example the case $\pi=0$.)
\end{remark} 

\subsection{Moser flows and bisections} \label{subsec:moser1}
Given a symplectic groupoid $\G\rra M$ integrating the Poisson manifold $(M,\pi)$, the Moser method 
has the following interpretation in terms of bisections. 
\begin{proposition}\label{prop:groupoid}
Suppose that $(\G,\om)$ is a symplectic groupoid integrating 
$(M,\pi)$, and let $\phi_t$ be a family of bisections of $\G$, 
with $\phi_0=\id$. Then there is a family of 1-forms $a_t$ satisfying 
$\d a_t=-\f{\p}{\p t} \sig_{\phi_t}$, and with the following properties:
\begin{enumerate}
\item $a_t$ is a Moser 1-form for the family of Poisson structures 
$\pi_t=(\A(\phi_t)^{-1})_*\pi$ on $M$, with $F_t=\A(\phi_t)^{-1}$ as the Moser flow. 
\item $\tz^* a_t$ is a Moser 1-form for the family of Poisson structures 
$\pi_{\G,t}=(\A^L(\phi_t)^{-1})_*\pi_\G$ on $\G$, with 
$F_{\G,t}=\A^L(\phi_t)^{-1}$ as the corresponding Moser flow. 
\end{enumerate}
Conversely, suppose $\sig_t$ is a family of closed 2-forms on $M$, with $\sigma_0=0$, defining 
gauge transformations $\pi_t=\pi^{\sig_t}$ of the Poisson structure on on $M$. Suppose $a_t$ is 
a Moser 1-form.  Then $\tz^* a_t$ are Moser 1-forms for the family of Poisson structures
$\pi_{\G,t}=\pi^{\tz^*\sig_t}_\G$ on $\G$. If the corresponding Moser vector field is complete, defining a flow 
$F_{\G,t}$, then $\phi_t=(F_{\G,t})^{-1}\circ \iz$ defines bisections with 
\[ \sig_{\phi_t}=\sig_t.\]
\end{proposition}
\begin{proof}
Given $\phi_t$, let $\ti{v}_t$ be the time dependent vector field generating the flow
$t\mapsto\A^L(\phi_t)^{-1}$, and put
\begin{equation}  \label{eq:omega+sigma}
\om_t=\A^L(\phi_t)^*\om=\om+\tz^*\sigma_{\phi_t}.
\end{equation}
Taking the $t$-derivative of the equation $\om=(\A^L(\phi_t)^{-1})^*\om_t$, we find 
\begin{equation}\label{eq:find}
\f{\p}{\p t}\om_t=\d \iota(\ti{v}_t)\om_t=-\d \wt{a}_t.\end{equation}
where 
\begin{equation}\label{eq:wtat}
 \wt{a_t}=-\iota(\ti{v}_t) \om_t\in \Om^1(\G).\end{equation}
%
%Let $a_t=\tz^* \wt{a}_t$. 
We claim that $\wt{a}_t$ is $\tz$-basic, so that 
$\wt{a}_t=\tz^* a_t$ for a family of 1-forms $a_t$ on $M$. Equation \eqref{eq:wtat} then shows that 
$\tz^*a_t$ is a Moser 1-form on $\G$ for the family of Poisson structures $\pi_{\G,t}$, 
with $\A^L(\phi_t)^{-1}$ as the corresponding Moser flow. Equation \eqref{eq:find} and 
$\f{\p}{\p t}\om_t=\tz^*\f{\p}{\p t}\sigma_{\phi_t}$ show that 
$\d a_t=-\f{\p}{\p t} \sig_{\phi_t}$. 

To prove the claim note that $\ti{v}_t$ is tangent to $\sz$-fibers, since $\A^L$ preserves the $\sz$-fibers.
Recall that the tangent spaces to $\sz$-fibers and $\tz$-fibers are $\om$-orthogonal. Since $\A^L$ acts by automorphisms 
of both the $\tz$-fibration and the $\sz$-fibration, the tangent spaces to $\sz$-fibers and $\tz$-fibers are also 
$\om_t$-orthogonal. It follows that $\wt{a}_t=-\iota(\ti{v}_t)\om_t$ vanishes on vectors tangent to $\tz$-fibers, i.e., it is $\tz$-horizontal. 
On the other hand, if $\psi$ is a bisection, 
\[ \begin{split}
\A^R(\psi)^*\wt{a_t}&=
- \iota(\ti{v}_t)\A^L(\phi_t)^*\A^R(\psi)^*\om\\
&= - \iota(\ti{v}_t)\A^L(\phi_t)^*(\om-\sz^*\sig_\psi)\\
&=\wt{a_t}-\iota(\ti{v}_t)\sz^*\A(\phi_t)^*\om=\wt{a_t}.
\end{split}\]
This proves the claim,  and hence completes the proof of part (b).
Letting $v_t$ be the time dependent vector field generating $t\mapsto \A(\phi_t)^{-1}$, we have $\ti{v}_t\sim_{\tz} v_t$, hence $v_t=-\pi_t^\sharp(a_t)$, which proves part (a). Note $\phi_t=\A^L(\phi_t)\circ \iz=(F_{\G,t})^{-1}\circ \iz$.

Suppose conversely that $\sig_t$ and $a_t$ are given, as in the second part of the proposition. Then $\tz^*\sig_t$ defines 
gauge transformations $\pi_{\G,t}$, with Moser 1-forms $\wt{a}_t=\tz^*a_t$. Its Moser vector field $\wt{v}_t$ 
is $\tz$-related to the Moser vector field $v_t$ of $a_t$; hence 
its flow $\wt{F}_t$ satisfies $\tz\circ \ti{F}_t=F_t\circ \tz$. Since $\wt{v}_t$ is $\A^R$-invariant, the flow $\wt{F}_t$ commutes with $\A^R$. It follows that $\wt{F}_t=\A^L(\phi_t)^{-1}$ 
where $\phi_t=\ti{F}_t^{-1}\circ \iz$. Since $\A(\phi_t)^*\om=\om+\tz^*\sig_{\phi_t}$, and since $\tz$ is a submersion, we have $\sig_t=\sig_{\phi_t}$.
\end{proof}

\begin{remark}
By a similar discussion,  $-\sz^*a_t$ is a Moser 1-form on $\G$ for the family of Poisson structures 
$\A^R(\phi_t)^{-1}_*\pi_\G$, with $\A^R(\phi_t)^{-1}$ as the corresponding Moser flow. 
The corresponding Moser vector field commutes with that for $\tz^*a_t$, since the flows
(given in terms of the actions $\A^L$ and $\A^R$) commute. This implies that 
$\tz^*a_t-\sz^* a_t$ is again a Moser 1-form, for the family of Poisson structures 
$\Ad(\phi_t)^{-1}_*\pi_\G$ with $\Ad(\phi_t)$ as the Moser flow.
\end{remark}

Proposition \ref{prop:groupoid} extends to the equivariant case: Consider the setting from the end of 
Section \ref{subsec:moser}; in particular $ e_M(\xi)\sim_{(\Phi_t,\sig_t)}e_{\g^*}(\xi)$ and the Moser 1-forms $a_t$ satisfy 
\eqref{eq:ateqn}. Then the bisections $\phi_t$ satisfy 
\begin{equation}\label{eq:withmommap}
 (\Phi_t,\sig_t)\circ (\A(\phi_t),\sig_{\phi_t})^{-1}=(\Phi_0,0).
\end{equation}
Indeed, this identity is equivalent to the two conditions $\sigma_t=\sigma_{\phi_t}$ and $F_t^*\Phi_t=(\A(\phi_t)^{-1})^*\Phi_t=\Phi_0$.
%Note that pull-back $\t^*\Phi$ is the moment map for a Hamiltonian $\g$-action on $\G$: 
%\[ \iota(\xi_\G)\om=-\tz^*\l\d\Phi,\xi\r.\]
%
\begin{remark}
It is clear that this proposition, and its proof, has a local counterpart (working with local groupoids, local flows and so on), 
as well as a version for germs. 
\end{remark}

\begin{remark}
Use the symplectic form $\om$ on $\G$ to identify $A\G\cong T^*M$.  
Informally, the group $\Gamma(\G)$ of bisections may be viewed as the infinite-dimensional Lie group 
integrating the Lie algebra $\Gamma(A\G)\cong \Gamma(T^*M)$. From this perspective, the time dependent bisection $\phi_t\in\Gamma(\G)$ is 
the integration of the time dependent section $a_t\in\Gamma(T^*M)$.
\end{remark}

\subsection{Linearization}\label{subsec:lin}
Given a manifold $M$, we denote by $C^\infty(M)_x$ the algebra of germs at $x$ of smooth functions. Thus 
$C^\infty(M)_x=\on{colim}_U C^\infty(U)$, where $U$ ranges over open neighborhoods of $x$. Similarly, one defines
the space $C^\infty(M,N)_x$ of germs at $x$ of smooth functions to
another manifold $N$, germs of sections of fiber bundles, and 
so on. 
Given $F\in C^\infty(M,N)_x$, we will write $F\colon M_x \to N_y$ if $F(x)=y$. 
Suppose $\pi_M$ is a germ of Poisson structure at $x\in M$, given by a Poisson bracket on the algebra $C^\infty(M)_x$.  
Given another germ of a Poisson structure $\pi_N$ at $y\in N$, 
we will say that $F\colon M_x\to N_y$ is Poisson if the map $F^*\colon C^\infty(N)_y\to C^\infty(M)_x$ preserves brackets. 
If a Poisson structure $\pi_M$ vanishes at $x\in M$, then the tangent space $T_xM$ acquires a linear Poisson structure.

\begin{definition}
 Let $(M,\pi_M)$ be a Poisson manifold, and $x\in M$ a zero of $\pi_M$. Then $M$ is called \emph{linearizable at $x$}
if there exists a germ of a Poisson diffeomorphism
\begin{equation}\label{eq:linearization}
 F\colon (T_xM)_0\to M_{x},
\end{equation}
with $T_0F$ the identity transformation of $T_xM$.
We will refer to $F$ as a 
\emph{Poisson linearization}.
\end{definition}
\begin{remark}
\begin{enumerate}
\item
There are similar definitions in the \emph{formal} category, working 
with infinite jets of functions rather than germs, and in the 
\emph{analytic} category, requiring that the given data are analytic and working with germs of analytic functions.  
\item
There are analogous notions of linearizations of Lie algebroids, 
with the Poisson case corresponding to the cotangent Lie algebroid. 
See Fernandes-Monnier \cite{fer:lin} for a survey and some recent results. 
\end{enumerate}
\end{remark}

Suppose $\pi_M$ vanishes at $x$, and that we have found a germ of a $\sigma$-twisted Poisson linearization 
\[ (F,\sigma)\colon (T_xM)_0\to M_x.\] 
Recall that $T_xM=\k^*$ with the Lie--Poisson structure, 
where $\k$ is the isotropy Lie algebra at $x$. Let $s_t\colon \k^*\to \k^*$ be scalar multiplication by $t$, 
and define a family of closed 2-forms $\sigma_t=t^{-1}s_t^*\sigma_t$. Let $a_t$ be a family of  primitives 
of $-\f{\p}{\p t}{\sigma}_t$ (e.g., given by the de Rham homotopy operator for $\Om(\k^*)$). The Moser method 
determines a family of bisections $\phi_t$ such that $\phi_t(0)=e$ and $\sig_{\phi_t}=\sig_t$. Put $\phi=\phi_1$, so that 
$\sig_\phi=\sigma$. Then 
\[ (F,\sigma)\circ (\A(\phi),\sigma_\phi)^{-1}=(F\circ \A(\phi)^{-1},0),\] 
so that $F\circ \A(\phi)^{-1}$ is the desired Poisson linearization.

\subsection{Coisotropic submanifolds}

Recall that a submanifold $C\subset M$ of a Poisson manifold is 
\emph{coisotropic} if $\pi^\sharp(\on{ann}(TC))\subset TC$. Equivalently, the conormal bundle $\on{ann}(TC)\subset T^*M$ is a 
Lie subalgebroid. At least locally, this Lagrangian Lie subalgebroid 
integrates to a Lagrangian Lie subgroupoid $\L\subset \G$; see Cattaneo \cite{cat:int} and Xu \cite{xu:poigr}. Conversely, the base of any Lagrangian Lie subgroupoid $\L\rra C$ of $\G\rra M$ is a coisotropic submanifold. 
\begin{proposition}\label{prop:coisotropic}
Suppose that $C\subset M$ is a coisotropic submanifold of the Poisson manifold $(M,\pi)$, and 
that the Moser 1-forms $a_t$ (hence also the 2-forms $\sig_t$) pull back to $0$ on $C$. Then the Moser vector field $v_t$ is tangent to $C$. Furthermore, if $C\subset M$ integrates to a (local) Lagrangian subgroupoid of the (local) symplectic groupoid $\G$, then the
bisections $\phi_t\in \Gamma(\G)$ obtained from $a_t$ restrict to bisections of the Lagrangian Lie subgroupoid $\L$. 
\end{proposition}
\begin{proof}
By assumption, $a_t$ restricts to a section of the conormal bundle $\on{ann}(TC)$. Since the 
bundle automorphism  $1+\sig_t^\flat\,\pi^\sharp$ of $T^*M$ 
restricts to identity on $\on{ann}(TC)$, we see that 
$C$ is also coisotropic with respect to 
$\pi_t$. We claim that $\L$ is also Lagrangian with respect to 
$\om_t=\om+\tz^*\sigma_t$. Indeed, if $Y_1,Y_2\in T_x\L$, then their projections under $T\t$ are tangent to $C$, hence 
$(\tz^*\sigma_t)(Y_1,Y_2)=0$. Let $\ti{v}_t$ be the vector field 
defined by $\iota(\ti{v}_t)\omega_t=-\tz^* a_t$. For all $Y\in T_x\L$ we have 
\[ \om_t(\ti{v}_t|_x,Y)
=-\langle (T_x\tz)^* a_t|_x, Y\rangle
=- \langle a_t|_x,\ (T_x\tz)(Y)\rangle=0\]
since $(T_x\tz)(Y)\in T_xC$. This shows that $X_t$ is 
$\om_t$-orthogonal to the tangent spaces of $\L$, hence it is itself tangent to $\L$ and
its flow $\wt{F}_t$ preserves 
$\L$. We conclude that the bisections $\phi_t=\wt{F}_t^{-1}\circ \iz$ 
restrict to bisections of $\L$ over $C$. 
\end{proof}
This result implies the following functorial aspects of Moser flows.
Recall that a map $f\colon M_1\to M_2$ between Poisson manifolds
is Poisson if and only if its graph is a coisotropic submanifold 
\[ C=\on{Gr}(f)\subseteq M=M_2\times \ol{M_1},\] 
where $\ol{M_1}$ signifies $M_1$ with the opposite Poisson structure $-\pi_1$. Hence, if $\G_i$ are the (local) symplectic groupoids integrating $M_i$, we obtain a (local) Lagrangian subgroupoid 
$\L\subset \G_2\times\ol{\G_1}$ over $C$. This $\L$ is the graph of a \emph{co}morphism of Lie groupoids \cite{hig:dua}; 
it integrates the Lie algebroid comorphism between $T^*M_1$ and $T^*M_2$. (See Cattaneo-Dherin-Weinstein \cite{cat:int1} for a general discussion of 
the integration of comorphisms.) In particular we have a 
pull-back map of (local) bisections, integrating the pull-back map for 1-forms. Applying Proposition \ref{prop:coisotropic} to this situation, we obtain: 
\begin{proposition} \label{prop:twobisections}
Suppose $f\colon (M_1,\pi_1)\to (M_2,\pi_2)$ 
is a Poisson map, and let $a_{i,t}$ be families of 1-forms on 
$M_i$, defining Moser vector fields $v_{i,t}$. If 
$a_{1,t}=f^* a_{2,t}$, then $v_{1,t}\sim_f v_{2,t}$. 
Hence, $f$ intertwines the Moser flows. Furthermore, if 
$\G_i$ are the (local) symplectic groupoids integrating $M_i$, 
then the (local) bisections $\phi_{i,t}\in\Gamma(\G_i)$ are related by pull-back.  
\end{proposition}
\begin{proof}
This is a direct application of Proposition \ref{prop:coisotropic}
to $M=M_2\times \ol{M_1}$, with $a_t=(a_{2,t},-a_{1,t})$. 
\end{proof}

\section{Coboundary Poisson Lie groups}\label{sec:cob}
We review some aspects of the theory of Lie bialgebras and Poisson Lie
groups, due to Drinfeld \cite{dr:ha,dr:qu,dr:quas}.

\subsection{Lie bialgebras}
A \emph{Lie bialgebra} is a Lie algebra $\g$ together
with a linear map $\lambda\colon \g\to \g\otimes\g$ such
that 
\begin{enumerate}
\item[(i)] the map $\lambda$ defines a Lie coalgebra structure (i.e., it
defines a Lie bracket on $\g^*$), and 
\item[(ii)] $\lambda$ is a Lie algebra 1-cocycle with coefficients in the $\g$-module $\g\otimes\g$, i.e., 
\[ \lambda([\xi,\eta])
=\ad_\xi\lambda(\eta)-\ad_\eta\lambda(\xi)\] 
for all $\xi,\eta\in\g$.  
\end{enumerate}
The
structure of a Lie bialgebra is equivalent to that of a \emph{Manin
triple} $(\dd,\g,\h)$, given by a Lie algebra $\dd$ with a
non-degenerate invariant metric $\l\cdot,\cdot\r$ and a pair of
transverse Lagrangian Lie subalgebras $\g,\h$. The pairing identifies
$\g^*\cong \h$, and $\g$ acquires a Lie bialgebra structure, with
cobracket dual to the Lie bracket on $\h$. The Lie algebra $\dd$ is
called the \emph{double} of the Lie bialgebra $\g$. In the special case 
$\lambda=0$, the double is the  semi-direct product $\dd=\g\ltimes\g^*$ with respect to the 
coadjoint action, with $\h=\g^*$.  

A Lie bialgebra $\g$ is called a \emph{coboundary Lie
bialgebra} if $\lambda$ is the coboundary of some element $r\in
\g\otimes\g$, that is, 
\[ \lambda(\xi)=\ad_\xi r\]
for all $\xi\in \g$. The choice of $r$ for a given $\lambda$ will be referred to as a \emph{coboundary structure} for the Lie bialgebra $\g$.   
For any $r\in \g\otimes\g$, let $r^\sharp\colon \g^*\to \g$ denote the map $r^\sharp(\mu)=r(\mu,\cdot)$.

\begin{lemma}[Drinfeld \cite{dr:quas}] 
Let $\g$ be a Lie bialgebra, with double $\dd=\g\oplus \g^*$. There is a 1--1 correspondence 
between 
\begin{itemize}
\item
coboundary structures $r\in\g\otimes\g$, 
\item $\g$-equivariant splittings 
$j \colon \g^* \to \dd$
of the sequence
\[ 0\to \g\to \dd\to \g^*\to 0.\]
\end{itemize}
Under this correspondence, $j(\mu)= \mu - r^\sharp(\mu)$ 
for $\mu\in\g^*$.  
\end{lemma}
\begin{proof}
Let $r\in\g\otimes\g$ and $j\colon\g^*\to \dd$ be related 
as above. The map $j$ is $\g$-equivariant if and only if 
the expression
\[ 
\begin{array}{lll}
A(\xi, \mu) & = & [\xi, j(\mu)]_\dd - j(\ad^*_\xi \mu)  \\
& = & [\xi, \mu - r^\sharp(\mu)]_\dd - \ad^*_\xi \mu + r^\sharp(\ad^*_\xi\mu) \\
& = & -\ad^*_\mu\xi - [\xi, r^\sharp(\mu)]_\g + r^\sharp(\ad^*_\xi\mu)
\end{array}
\]
vanishes for all $\xi \in \g$ and $\mu \in \g^*$. Note that $\A(\xi, \mu) \in \g$.
Hence its vanishing is equivalent to the vanishing of 
\[
\begin{array}{lll}
\l \nu, A(\xi, \mu) \r & = & \l \nu, - \ad^*_\mu \xi - [\xi, r^\sharp(\mu)]_\g + r^\sharp(\ad^*_\xi \mu) \r \\
& = & \l [\mu, \nu]_{\g^*} , \xi \r + r(\ad^*_\xi \mu, \nu) + r(\mu, \ad^*_\xi \nu) \\
& = & \l \mu \otimes \nu, \lambda(\xi) - \ad_\xi r \r .
\end{array}
\]
for all $\xi \in \g$ and $\mu, \nu \in \g^*$. 
Hence, $\g$-equivariance of $j$ is equivalent to $\lambda(\xi)=\ad_\xi r$, as required.
%For all $\xi\in\g$ and $\mu\in\g^*$, the difference $[\xi,\mu]_\dd-\ad^*_\xi\mu$ lies in $\g$. 
%Applying $f$, this shows 
%
%\[[\xi,\mu]_\dd-\ad^*_\xi\mu=
%f([\xi,\mu]_\dd)-f(\ad^*_\xi\mu)=
%f([\xi,\mu]_\dd)+r^\sharp(\ad^*_\xi\mu).\]
%
%The map $f$ is $\g$-equivariant if and only if for all $\xi\in\g$ and $\mu\in\g^*$,
%$f([\xi,\mu]_\dd)
%=\ad_\xi f(\mu)=-\ad_\xi r^\sharp(\mu)$, that is, 
%
%\[
%[\xi,\mu]_\dd=r^\sharp(\ad^*_\xi\mu)-\ad_\xi r^\sharp(\mu)+\ad^*_\xi\mu.
%\]
%
%The $\g^*$-component of this equality follows from the
%definition of the Lie bracket on $\dd$.
%To analyse the $\g$-component, replace $\mu$ with $\mu_1$, and pair both
%the left and the right hand side with $\mu_2\in\g^*$
%to obtain
%\[ \l \xi,[\mu_1,\mu_2]_\dd\r
%=r(\ad^*_\xi\mu_1,\mu_2)+r(\mu_1,\ad^*_\xi\mu_2)
%=(\ad_\xi r)(\mu_1,\mu_2),\]
%i.e.~ $\lambda(\xi)=\ad_\xi r$.  
\end{proof}
Note that  the image of the inclusion $\mf{p} = j(\g^*) \subset \dd$ 
is a $\g$-invariant complement to $\g$ in $\dd$.

\begin{remark}
The set of coboundary structures for a Lie bialgebra $\g$ is either
empty, or is an affine space with underlying vector space
$(\g\otimes\g)^\g$.  In terms of the splitting
$j\colon \g^*\to \dd$, this follows since any two $\g$-equivariant splittings differ
by an element of
\[ \Hom_\g(\g^*,\g)\cong (\g\otimes\g)^\g.\] 
%
%To see this, note that the $\g$-equivariant projections 
%$f\colon \dd\to \mf{p}$ are an affine subspace 
%of $\Hom_\g(\dd,\g)$, with underlying vector space the homomorphisms vanishing on $\g$. 
The affine space of coboundary structures is preserved 
under the involution $\sigma: \xi_1\otimes\xi_2\mapsto \xi_2\otimes\xi_1$ of $\g\otimes\g$. 
In terms of splittings it amounts to replacing $j: \mu \mapsto \mu - r^\sharp(\mu)$ with 
$\tilde{j}: \mu \mapsto \mu + \sigma(r)^\sharp(\mu)$.
In terms of complements, the involution takes $\mf{p}$ 
to $\mf{p}^\perp$.
Halfway between $\mf{p}$ and $\mf{p}^\perp$, one hence finds an invariant \emph{Lagrangian} complement. It corresponds to the choice of a \emph{skew-symmetric} $r$ (that is, $\sigma(r)=-r$).
Note that if $\mf{p}$ is an invariant Lagrangian complement, then 
\[ [\g,\mf{p}]\subseteq \mf{p},\ [\mf{p},\mf{p}]\subseteq \g\]
so that $(\dd,\g)$ is a symmetric pair. Let $\kappa$ be the  Lie algebra 
automorphism of $\dd$, equal to $+1$ on $\g$ and to $-1$ on $\mf{p}$. It exponentiates to a Lie group automorphism $\kappa$ of the simply connected Lie group $D$ integrating $\dd$. The identity component of its fixed point set is a \emph{closed} 
Lie subgroup $G\subset D$ integrating $\g$. Similarly, the identity component of the fixed point set of 
the anti-involution $d \mapsto \kappa(d)^{-1}$ 
%$\kappa\circ \on{Inv}_D$ 
of $D$ is a closed submanifold $P\subseteq D$ with $\exp(\mf{p})\subseteq P$.  
\end{remark}

\begin{remarks}
\begin{enumerate}
\item Suppose $\g$ is a Lie algebra, and $r\in \g\otimes\g$. Then $r$ defines a coboundary Lie bialgebra structure on 
$\g$ if and only if the  
symmetric part of $r$, $r+r_{21}\in \g^{\otimes 2}$ and the  
element $\on{YB}(r)=[r_{12},r_{13}]+[r_{12},r_{23}]+[r_{13},r_{23}]\in \g^{\otimes 3}$ are both $\g$-invariant.  (A proof, as well as an explanation of the notation, 
may be found in  \cite[Chapter 8]{maj:fou}.) 
The coboundary structure $r$, or the Lie bialgebra $\g$ itself, 
is called \emph{quasi-triangular} if $\on{YB}(r)=0$ (classical Yang-Baxter equation), and \emph{factorizable} if
furthermore the symmetric part $r+r_{21}$ is non-degenerate.  
As noted in \cite{dr:quas}, 
a quasi-triangular structure on $\g$ is equivalent 
to $j$ being a morphism of Lie algebras, i.e., 
to $\mf{p}=j(\g^*)$ being 
a complementary \emph{ideal} to $\g$ in $\dd$. 
\item 
If $\lambda$ defines a Lie bialgebra structure on 
$\g$, then so does $\lambda_t=t\lambda$ for all 
$t\in\R$. Hence, any Lie bialgebra structure may be regarded as a deformation of the trivial one. Coboundary structures may be scaled accordingly as $r_t=t r$.   
\item 
Every non-abelian Lie algebra admits a non-trivial coboundary Lie bialgebra structure \cite{sme:exi}. 
\end{enumerate}
\end{remarks}

\subsection{Poisson Lie groups}
A \emph{Poisson Lie group} is a Lie group $G$ together with a Poisson
structure such that the group multiplication is a Poisson map. The
Poisson tensor of a Poisson Lie group vanishes at the group unit,
hence $\g=T_eG$ acquires a linear Poisson structure, defining a Lie bracket on $\g^*$.
The
compatibility with the given Lie bracket is such that $\g$ becomes a
Lie bialgebra. Drinfeld's theorem \cite{dr:ha} gives a bijective
correspondence between Lie bialgebras and simply connected Poisson Lie groups. A Poisson Lie group will be called 
coboundary if its tangent Lie bialgebra is of this type. 

\begin{example}[Lu-Weinstein Poisson structure] Let $G$ be a compact 
Lie group, and $D=G^\C$ its complexification, regarded 
as a real Lie group. Choose a maximal torus $T\subset G$ and a 
system of positive roots, and consider the decomposition 
\[ \dd=\g\oplus \a\oplus\n\]
where $\n\subset \g^\C$ is the sum of positive root spaces, and 
$\a=\sqrt{-1}\t$. On the group level, this is the Iwasawa decomposition 
$D=GAN$. Let $B$ be a positive definite invariant symmetric bilinear form on $\g$, and let the bilinear form $\l\cdot,\cdot\r$ on $\dd$ be 
the imaginary part of its complexification $B^\C$. Then $(\dd,\g,\h)$ with 
$\h=\a\oplus\n$ is a Manin triple, defining the \emph{Lu-Weinstein} Poisson structure on $G$. The subspace $\mf{p}=\sqrt{-1}\g$ is a $\g$-invariant Lagrangian complement to $\g$ in $\dd$, and $D=GP$ by the Cartan decomposition. Hence $\g$ is a coboundary Lie bialgebra. The corresponding coboundary structure $r$ is given in terms of  root vectors $e_\alpha$ (normalized such that $e_{-\alpha}=e_\alpha^*$ and $B(e_\alpha,e_{-\alpha})=2$) by 
\[r=\f{i}{2} \sum_{\alpha>0} (e_{-\alpha}\otimes e_{\alpha}-
e_{\alpha}\otimes e_{-\alpha})\] where the sum is over the set of positive roots.  
\end{example}

\subsection{$G^*$-valued moment maps}
Let $(\dd,\g,\h)$ be a Manin triple, $D$ a Lie group integrating $\dd$, and let $G$ and 
$H$ be  Lie subgroups (not necessarily closed), integrating $\g$ and $\h$. We will refer to $(D,G,H)$ as a \emph{Manin triple of 
Lie groups} integrating $(\dd,\g,\h)$. The Lie algebra $\dd$ acts on $H=G^*$ by the (left)
\emph{dressing action}, $\zeta\mapsto \zeta_{G^*}$ where 
\[ \iota(\zeta_{G^*})\theta^L_{G^*}|_h=-\pr_{\g^*}(\Ad_{h^{-1}}\zeta),\]
Here $\pr_{\g^*}\colon\dd\to \g^*$ is the projection to the first summand in 
$\dd=\h\oplus \g$. (If the multiplication map defines a global diffeomorphism $G^*\times G\to D$, then this 
infinitesimal action integrates to the natural $D$-action on $G^*$ regarded as a homogeneous space $G^*=D/G$.) 
One knows (see e.g., \cite{lu:mo}) that the symplectic leaves of $G^*$ are exactly the orbits of this dressing action of $\g\subset \dd$.
The cotangent Lie algebroid $T^*G^*\cong \on{Gr}(\pi_{G^*})$ has the structure of an action Lie algebroid for this action.
In fact, there is a bracket preserving linear map $e_{G^*}\colon \dd\to \Gamma(\T G^*)$ given by 
\[ e_{G^*}(\zeta)=(\zeta_{G^*},-\l\theta^R_{G^*},\zeta\r)\in \Gamma(\T G^*),\]
and $\on{Gr}(\pi_{G^*})$ is spanned by the sections $e_{G^*}(\xi)$ for $\xi\in\g$.
Given a Poisson manifold $(M,\pi)$ and a Poisson map 
$\Psi\colon M\to G^*$, one obtains unique sections $e_M(\xi),\ \xi\in\g$ of  
$\on{Gr}(\pi)\subset \TM$ such that  $e_M(\xi)\sim_{(\Psi,0)} e_{G^*}(\xi)$. The vector field component $\xi_M$ of $e_M(\xi)$ 
defines a $\g$-action on $M$, with
\begin{equation}\label{eq:actionformula}
\xi_M=-\pi_M^\sharp \l\Psi^*\theta^R_{G^*},\xi\r.\end{equation}
This action is a Lie bialgebra action, 
with $\Psi$ as its moment map, in the sense of J.-H. Lu \cite{lu:mo}. 
For example, the identity map from $G^*$ to itself is a moment map for the dressing action, while the inclusion of dressing orbits is a moment map for the action on these orbits. 

The Lie group $D$  is itself a Poisson Lie group, with Manin triple 
\[ (\dd\oplus \ol{\dd},\dd_\Delta,\h\oplus \g).\] 
Here $\ol{\dd}$ is equal to $\dd$ as a Lie algebra, but with the opposite metric, and $\dd_\Delta\subset \dd\oplus \ol{\dd}$ is a copy of $\dd$ embedded diagonally. One refers to $D$ with this Poisson structure as a \emph{Drinfeld double} of $G$.
The anti-diagonal in $\dd\oplus \ol{\dd}$ is a Lagrangian $\dd_\Delta$-invariant complement, defining a 
coboundary structure on $D$. 
The inclusion $\g\to \dd\cong \dd_\Delta$ is a morphism of Lie bialgebras, defining a morphism of Poisson Lie groups 
$G\hra D$. The dual map $\dd^*=\h\oplus \g \to \g^*=\h$ is simply projection to the first factor. Let us 
describe the resulting morphism of Poisson Lie groups $\pr_{G^*}\colon D^*\to G^*=H$.
Note first that $D^*$ is the Poisson Lie group
\[ D^*=H\times G=G^*\times G,\]
with Poisson structure defined by the Manin triple $(\dd\oplus \ol{\dd},\h\oplus \g,\dd_\Delta)$.
Consider the dressing action of $\dd\oplus \ol{\dd}$ on $D^*$. Given $(\zeta,\zeta')\in \dd\oplus\ol{\dd}$ we have 
\[\pr_{\h\oplus\g}(\zeta,\zeta')=(\pr_\h(\zeta-\zeta'),\pr_\g(\zeta'-\zeta)),\] 
hence
\begin{equation}\label{eq:dressstar}
((\zeta,\zeta')_{D^*})(\theta^L_H+\theta^L_G)=\big(-\pr_\h(\Ad_{h^{-1}}\zeta-\Ad_{g^{-1}}\zeta'),\ 
\pr_\g(\Ad_{h^{-1}}\zeta-\Ad_{g^{-1}}\zeta')\big).
\end{equation}
The graph of $\pi_{D^*}$ is thus spanned by the sections $e_{D^*}(\zeta):=e_{D^*}(\zeta,\zeta)$ for $\zeta\in\dd$, where 
\[ e_{D^*}(\zeta,\zeta')=\big((\zeta,\zeta')_{D^*},\ -\l\theta^R_H,\zeta\r+\l\theta^R_G,\zeta'\r\big)\]
for $(\zeta,\zeta')\in\dd\oplus \ol{\dd}$. We see that $e_{D^*}(\xi)\sim_{(\pr_{G^*},0)}e_{G^*}(\xi)$
for $\xi\in\g$. 

\subsection{The Lu-Weinstein double groupoid}\label{subsec:doublegroupoid}
According to Lu-Weinstein \cite{lu:gr}, any Poisson Lie group is integrated by a symplectic double groupoid $\G$. 
If $(D,G,G^*)$ is a Manin triple of Lie groups integrating $(\dd,\g,\g^*)$, one has 
\[ \G=\{(u,g,g',u')\in G^*\times G\times G\times G^*|\ ug=g'u'\}.\]
%
%The symplectic form $\omega$ is the pull-back of 
%\[ -\hh \l u^{-1}\d u,\ \d g g^{-1}\r-\hh \l \d u' (u')^{-1},\ (g')^{-1}\d g'\r\]
%to the submanifold given by $ug=g'u'$; here $u^{-1}\d u$ indicates the pull-back of $\theta^L_{G^*}$ under the 
%map $(u,g,g',u')\mapsto u$, similarly for $\d g g^{-1}$ is the pull-back of $\theta^L_G$, and so on. 
As a groupoid over $G^*$, the source and target map take $(u,g,g',u')$ to $u'$ and $u$, respectively, and the groupoid multiplication 
of composable elements reads as
\[ (u_1,g_1,g_1',u_1') (u_2,g_2,g_2',u_2')=(u_1,g_1g_2,g_1'g_2',u_2').\]
See \cite{lu:gr} for a description of the symplectic structure, and further details.
The map $\G\to D^*,\ (u,g,g',u')\mapsto (u,(g')^{-1})$ is Poisson, hence it is a moment map for an 
action $\dd\to \mf{X}(\G),\ \zeta\mapsto \zeta_\G$. The map   $\G\to D,\ (u,g,g,u')\mapsto ug$ intertwines $\zeta_\G$ with $-\zeta^R$.

\section{Linearization of Poisson Lie group structures}\label{sec:linearization}
In this section, we will prove that the Poisson structure on the the dual of a coboundary Poisson Lie group $G^*$ is linearizable at the group unit. Our argument will depend on the existence of a germ of an equivariant map $\Exp\colon \g^*\to G^*$ and a germ of a 
closed 2-form $\sigma$ on $\g^*$ such that the pair 
$(\Exp,\sigma)$ gives a twisted Poisson linearization $(\g^*,\pi_{\g^*})$ to $(G^*,\pi_{G^*})$. 
The construction will involve a twisted Poisson linearization of $D^*$, the dual of the Drinfeld double. 

Throughout this section, we take $(\dd,\g,\h)$ to be a  Manin triple, and $(D,G,H)$ a corresponding Manin triple of Lie groups.

\subsection{The Dirac Lie group structure of $D$}
As we recalled earlier, the Drinfeld double $D$ of a Poisson Lie group is itself a Poisson Lie group. Using only the invariant metric on $\dd$, it also has another structure as a \emph{Dirac Lie group}, as follows. Let $\zeta_D=\zeta^L-\zeta^R,\ \zeta\in\dd$ be the vector fields generating the conjugaction action, denote by $\theta^L$ and $\theta^R\in \Om^1(D,\dd)$ the left-invariant and the right-invariant Maurer-Cartan forms on $D$, and let 
\[ e_D(\zeta,\zeta')=\Big((\zeta')^L-\zeta^R,-\hh\l\theta^L,\zeta'\r-\hh\l\theta^R,\zeta\r\Big)\in \Gamma(\T D).\]
for $\zeta,\zeta'\in\dd$. The sections $e_D(\zeta):=e_D(\zeta,\zeta)$ span the so-called \emph{Cartan-Dirac structure}. It is 
a Dirac structure for a modified Courant bracket $\Cour{\cdot,\cdot}_\eta$ on $\Gamma(\T D)$, where $\eta$ is the 
Cartan 3-form 
\[ \eta=\f{1}{12}\l \theta^L,[\theta^L,\theta^L]\r\in\Om^3(D).\]
For further details, see e.g., \cite{cou:di,al:pur}. The Cartan 3-form has the following property
\begin{equation}\label{eq:contrvarpi}
\iota(\zeta_D)\eta=-\hh \d \l \theta^L+\theta^R,\zeta\r,\ \ \zeta\in\dd.
\end{equation}
Let $\Mult\colon D\times D\to D$ be the group multiplication, and 
\[ \varsigma=\hh\, \l \pr_1^*\theta^L,\,\pr_2^*\theta^R\r\in \Om^2(D\times D).\]
Then $(\Mult,\varsigma)$ is a morphism of Dirac structures. In fact, 
\begin{equation}\label{eq:Mult} \Mult^*\eta =\pr_1^*\eta+\pr_2^*\eta-\d\varsigma
\end{equation}
and $(e_D(\zeta),e_D(\zeta))\ \sim_{(\Mult,\varsigma)}\ e_D(\zeta)$.

\subsection{Poisson linearization of $D^*$}
Using the metric, we may identify $\dd$ with $\dd^*$; hence we have the sections $e_\dd(\zeta)=(\zeta_\dd,-\l\d\mu,\zeta\r),\ \zeta\in\dd$ as in Section \ref{subsec:ham}. Let $\varpi\in \Om^2(\dd)$ be the image of $\exp^*\eta$ under the standard homotopy operator $\Om^q(\dd)\to \Om^{q-1}(\dd)$, for the linear retraction onto the origin. An explicit formula for $\varpi$ may be found in \cite[Appendix C]{me:clifbook}.  
\begin{proposition}\cite{al:pur} The morphism $(\on{exp},\varpi)$ satisfies
\[ e_\dd(\zeta)\sim_{(\exp,\varpi)}e_D(\zeta)\]
for all $\zeta\in\dd$. Away from the critical points of $\exp$, the morphism  $(\exp,\varpi)$ is a Dirac morphism, for the Lie-Poisson structure on $\dd=\dd^*$ and the 
Cartan-Dirac structure on $D$.
\end{proposition}
We may think of $(\exp,\varpi)$ as defining a twisted Dirac linearization of the Dirac Lie group $D$. We will use this to obtain a twisted Poisson linearization of the Poisson Lie group $D^*$.
%In particular, the 2-form $\varpi$ is $\dd$-invariant, and its contractions with the generators for the conjugation action are, 
%\[ \iota(\zeta_\dd)\varpi=-\l\d\mu,\zeta\r+\hh\exp^* \l \theta^L+\theta^R,\zeta\r,\ \ \zeta\in\dd.\]
%See \cite{al:mom} or \cite{al:pur} for a proof. 

%
As it turns out, there is a close relation between the Cartan-Dirac structure on $D$ and the Poisson structure on $D^*$.
Define 
\[ F\colon D^*=H\times G\to D,\ (h,g)\mapsto hg^{-1},\]
and put 
\[ \varepsilon=\hh \l \theta^L_H,\theta^L_G\r\in\Om^2(D^*).\]
Note that $F$ has bijective differential, but is not necessarily surjective. 
\begin{proposition}
The morphism $(F,\varepsilon)\colon D^*\to D$ is a Dirac morphism, relative to the Poisson structure on $D^*$ and the Dirac Lie group 
structure on $D$. In fact  
\[ e_{D^*}(\zeta,\zeta')\sim_{(F,\varepsilon)} e_D(\zeta,\zeta')\]
for all $\zeta,\zeta'$. 
\end{proposition}

\begin{proof}
We have $F^*\eta=\d\varepsilon$, as a consequence of the formula for $\Mult^*\eta$ and the fact that 
$\eta$ pulls back to $0$ on the subgroups $H$ and $G$. Let us verify that the map $F$ intertwines the 
dressing action of $\dd\oplus \ol{\dd}$ with the usual action on $D$,
$(\zeta,\zeta')\mapsto (\zeta')^L-\zeta^R$. To see this, we will verify 
\[ F^* \iota((\zeta')^L-\zeta^R)\theta^L_D=\iota((\zeta,\zeta')_{D^*})F^*\theta^L_D.\]
At $d=F(h,g)=hg^{-1}$, 
\[ \Ad_{g^{-1}}\iota((\zeta')^L-\zeta^R)\theta^L_D=\Ad_{g^{-1}}\zeta'-\Ad_{h^{-1}}\zeta
.\]
On the other hand, $\Ad_{g^{-1}}(F^*\theta^L_D)=\theta^L_H-\theta^L_G$, 
hence the result coincides with \eqref{eq:dressstar}. 
Similarly, 
\[ \hh F^*(\l\theta^L_D,\zeta'\r+\l\theta^R_D,\zeta\r)=\hh \l\Ad_{g^{-1}}\zeta'+\Ad_{h^{-1}}\zeta,\ \theta^L_H-\theta^L_G\r.\]
while on the other hand 
\[ \begin{split}
\iota((\zeta,\zeta')_{D^*})\varepsilon&=\hh \l\Ad_{g^{-1}}\zeta'-\Ad_{h^{-1}}\zeta,\theta^L_G+\theta^L_H\r
\end{split}\]
The sum is $ -\l\zeta',\theta^R_G\r+\l\zeta,\theta^L_H\r$ as desired.  
\end{proof}

The map $F\colon D^*=H\times G\to D$ is a diffeomorphism on a neighborhood of the group unit. Hence, the germ of this map is invertible, and
\begin{equation}\label{eq:lindouble}
 (F,\varepsilon)^{-1}\circ (\exp,\varpi)=(F^{-1}\circ \exp,\varpi-\exp^* (F^{-1})^*\varepsilon)\colon \dd^*=\dd\to D^*\end{equation}
is a well-defined twisted Poisson map, on a neighborhood of  $0\in\dd^*$. To summarize, we have shown 
\begin{proposition} 
Let $(\dd,\g,\h)$ be a Manin triple, with corresponding Manin triple of Lie groups $(D,G,H)$.
Then \eqref{eq:lindouble} defines a twisted Poisson linearization of the 
dual Poisson Lie group $D^*=H\times G$.
\end{proposition}

\subsection{Twisted Poisson linearization of $G^*$}
Suppose now that $\g$ comes with a a coboundary structure, given by an $r$-matrix $r\in\g\otimes\g$ or equivalently by a 
$\g$-equivariant splitting $j\colon \g^*\to \dd$. The equivariance guarantees that  
$e_{\g^*}(\xi)\sim_{(j,0)} e_{\dd^*}(\xi)$
for all $\xi\in\g\subset \dd$, hence
\[ e_{\g^*}(\xi)\sim_{(j,0)}e_{\dd^*}(\xi)\sim_{(\exp,\varpi)}e_D(\xi)\sim_{(F,\varepsilon)^{-1}}e_{D^*}(\xi)\sim_{(\pr_{G^*},0)}e_{G^*}(\xi).\]
Let $(\Exp,\sigma)$ be the composition of these morphisms:
\[(\Exp,\sigma)= (\pr_{G^*},0)\circ (F,\varepsilon)^{-1}\circ (\exp,\varpi)\circ {(j,0)}.\]
%
%Using the multiplication map $(h,g)\mapsto hg$ to identify $D=G^*\times G$ (near $e$), we have that 
%\[ \Exp=\pr_{G^*}\circ \exp\circ j,\ \ \ \sigma=j^*(\varpi+\exp^*\l\theta^L_H,\theta^R_G\r).\]
%
Using the multiplication map $(h,g)\mapsto hg$ to identify $D=G^*\times G$ (near $e$), we have 
$\Exp(\mu)=\pr_{G^*}(\exp(j(\mu)))$. Note 
$T_0\Exp=\on{Id}_{\g^*}$. The notation is meant to suggest that 
we think of $\Exp$ as a replacement\footnote{If $j(\g^*)=\mf{p}$ is a $\g$-invariant Lagrangian complement, then the resulting map 
$\g^*\to D/G$ is indeed the standard notion of exponential map for a symmetric space.} for the exponential map, $\exp\colon \g^*\to G^*$.
The relation 
\begin{equation}\label{sigrelation}
 e_{\g^*}(\xi)\sim_{(\Exp,\sigma)} e_{G^*}(\xi)
 \end{equation}
says that $\Exp$ is $\g$-equivariant, and the closed 2-form $\sigma$ satisfies 
\begin{equation}\label{eq:contr}
 \iota(\xi_{\g^*})\sigma=\l\d\mu,\xi\r-\Exp^*\l \theta^R_{G^*},\xi\r.
\end{equation}
In particular, $(\Exp,\sigma)\colon (\g^*)_0\to (G^*)_e$ is a twisted Poisson linearization. As explained at the end of Section \ref{subsec:lin}, 
the Moser method guarantees the existence of a germ of a bisection $\psi\in\Gamma(T^*G)_0$ such that $\sigma_\psi=\sigma$. Then,   
$(\Exp,\sigma)\circ (\A(\psi),\sigma_\psi)^{-1}=(\Exp\circ \A(\psi)^{-1},0)$, and  we have shown:
\begin{theorem}\label{th:a}
Let $G^*$ be the dual Poisson Lie group of a coboundary Poisson Lie group $G$. Then there exists a germ of a bisection $\psi\in \Gamma(T^*G)_0$ 
of $T^*G\rra \g^*$ such that 
\[ \on{Exp}\circ \A(\psi)^{-1}\colon (\g^*)_0\to (G^*)_e\]
is a Poisson linearization. 
\end{theorem}

\begin{remarks}
\begin{enumerate}
\item 
The explicit choice of $\psi$ depends on the choice of a family of closed 2-forms $\sigma_t$ interpolating between $0$ and $\sigma$. 
As we shall see in Section \ref{subsec:resc} below, the choice $\sigma_t=t^{-1} s_t^*\sigma$, where
$s_t\colon \g^*\to \g^*$ is multiplication by $t$, has a nice geometric interpretation.

\item
For the special case that $G$ is a compact Lie group with the 
Lu-Weinstein Poisson structure, the Poisson diffeomorphism 
$\on{Exp}\circ \A(\psi)^{-1}$ is in fact globally defined (rather than just as a germ). Indeed, $\on{Exp}\colon \g^*\to G^*$ 
is a global diffeomorphism, as a consequence of the Iwasawa and Cartan decompositions $G^\C=GAN=GP$.
Furthermore, the Moser vector field used in the construction of $\psi$ is complete, due to compactness of the symplectic leaves of 
$\g^*$. This is the explicit Ginzburg-Weinstein diffeomorphism 
constructed in \cite{al:po} (see \cite{al:gw} for the formulation involving bisections). 

\item 
As mentioned in the introduction, the formal counterpart of 
this result was proved by Enriques-Etingof-Marshall in \cite{enr:co}, for a more restrictive class of Poisson Lie groups. Note however that the results in \cite{enr:co} are stated for arbitrary fields of characteristic zero. It should be possible to generalize our methods in that direction, using a formal version of the Moser argument. 
\end{enumerate}
\end{remarks}

\subsection{Linearization of the symplectic groupoid}
Let $\G_0=T^*G\rra \g^*$ and $\G\rra G^*$ be the symplectic groupoids of $\g^*$ and of $G^*$, respectively.
Proposition \ref{prop:integration} shows that the symplectic groupoid for the 
gauge transformed Poisson structure $\pi_{\g^*}^\sigma$ is $\G_0$, with the symplectic form modified by 
$\tz^*\sigma-\sz^*\sigma$. Since $\Exp\colon \g^*\to G^*$ is Poisson with respect to $\pi_{\g^*}^\sigma$, 
it lifts to a germ of an isomorphism $\G_0\to \G$. To make this explicit, let us regard $\G_0$ as the action groupoid 
$G\times\g^*$ for the coadjoint action, and $\G$ as the action groupoid $G\times G^*$ for the dressing action.
(In general, the latter is an identification of local groupoids, since the dressing action may not be globally defined.) 
Define a germ (at $(e,0)$) of a groupoid isomorphism
\[ \wt{\Exp}\colon \G_0\to \G,\  (g,\mu)\mapsto (g,\Exp\mu).\]
\begin{proposition}
We have the following equality of germs of 2-forms: 
\[ \wt{\Exp}^*\om_\G=\om_{\G_0}+\tz^*\sigma-\sz^*\sigma,\]
where $\sz,\tz$ are the source and target map of $\G_0$.
\end{proposition}
\begin{proof}
For $\xi\in\g$, let $\xi_\G$ be the vector field on $\G$ corresponding to the left dressing action (see Section \ref{subsec:doublegroupoid}), 
and similarly $\xi_{\G_0}$. These actions have the respective target maps of the groupoids $\G_0$ and $\G$ 
as moment maps:
\[ \iota(\xi_{\G_0})\om_{\G_0}=-\tz^*\l \d\mu,\xi\r,\ \ \ \ \iota(\xi_\G)\om_\G=-\tz^*\l\theta^R_{G^*},\xi\r.\]
Since the target map $\tz\colon \G_0\to \g^*$ is equivariant with respecti to the coadjoint action, we have 
\[ \iota(\xi_{\G_0})\tz^*\sigma=\tz^*\iota(\xi_{\g^*})\sigma=\tz^*\big( \l \d\mu,\xi\r-\Exp^*\l\theta^R,\xi\r\big).\]
The source map is equivariant relative to the trivial action which implies $\iota(\xi_{\G_0})\sz^*\sigma=0$. 
This shows that 
\[ \iota(\xi_{\G_0})\big( \wt{\Exp}^*\om_\G-\om_{\G_0}-\tz^*\sigma+\sz^*\sigma\big)=0.\]
The form inside the parentheses is closed and $\g$-horizontal. Hence, it is $\g$-basic and it suffices to show that its pull-back to  the unit bisection $\g^*\subset \G_0$ vanishes. 

Recall that the unit bisection of a symplectic groupoid is Lagrangian. Hence, the pull-back of $\om_{\G_0}$ to $\g^*$
vanishes. Under $\wt{\Exp}$, the unit section of $\G_0$ is mapped to the unit section of $\G$. Hence, the same
argument applies to the form $\wt{\Exp}^*\om_\G$. Finally, on the unit section $\sz^* \sigma = \tz^* \sigma = \sigma$
which completes the proof.

%But this is immediate, since the unit bisection of a symplectic groupoid is Lagrangian.
\end{proof}

\subsection{$G^*$ as a deformation of $\g^*$}\label{subsec:resc}
The bisection $\psi$ in Theorem \ref{th:a}, as obtained from the Moser method, depends on the choice of a family of closed 2-forms 
$\sigma_t$ interpolating between $0$ and $\sigma$. While it is possible to simply take $\sigma_t=t\sigma$, we will show that 
the choice 
\[ \sigma_t=t^{-1} s_t^*\sigma,\]
where $s_t\colon \g^*\to \g^*$ is multiplication by $t$, has the following interesting feature: 
the family of Moser 1-forms $a_t$ is given by a simple scaling law $a_t = t^{-2}s_t^* a_1$.
 
Let $(\dd,\g,\h)$ be a Manin triple, and $(D,G,H)$ a corresponding Manin triple of Lie groups. Let $\g_t$ be the 
Lie bialgebra, obtained from $\g$ by rescaling the cobracket by a factor $t\in\R$ (while keeping the Lie bracket unchanged). 
If $\g$ has a coboundary structure $r$, then $\g_t$ has a coboundary structure $r_t=t\, r$. 
We denote by $(\dd_t,\g,\h_t)$ the resulting family of Manin pairs, and by $D_t,\ H_t=G^*_t$ the Lie groups corresponding to $\dd_t$ 
and $\h=\g_t^*$. In particular, 
for $t=0$ we obtain the zero cobracket, with 
\[ \dd_0=\g\ltimes\g^*,\ \ G^*_0=\g^*,\ \ D_0=G\ltimes \g^*.\]
The bracket $[\cdot,\cdot]_t$ on $\dd_t=\g\oplus \g^*$ reads as 
\[ \begin{split}
[\xi_1,\xi_2]_t&=[\xi_1,\xi_2],\\ 
[\mu_1,\mu_2]_t&=t[\mu_1,\mu_2],\\ 
 [\xi,\mu]_t&=\ad^*_\xi \mu+t\, r^\sharp(\ad_\xi^*\mu)-t\ad_\xi r^\sharp(\mu),\end{split}\]
for $\xi,\xi_1,\xi_2\in\g$ and $\mu,\mu_1,\mu_2\in\g^*$. 
%Extend $s_t\colon \g^*\to \g^*$ to a map on 
%$\g\oplus\g^*$, acting trivially on the first summand. 
Let $\sigma_t\in \Om^2(\g^*)_0$ be the counterpart of the 2-form $\sigma$, and $\pi_t=\pi_{\g^*}^{\sig_t}$ the resulting gauge transformations of the Poisson structure. Denote by 
$a_t\in \Om^1(\g^*)_0$ the family of 1-forms obtained by applying the standard homotopy operator to $-\f{\p}{\p t}\sig_t$.
\begin{proposition}
The Poisson bivector fields $\pi_t$, the 2-forms $\sigma_t$,
the Moser 1-forms $a_t$ and the bisection $\phi_t$ scale according to 
\[ \pi_{t}=t\ s_t^* \pi_1,\ \  \sigma_t=t^{-1} s_t^*\sigma_1,\ \ 
a_t=t^{-2}s_t^* a_1, \ \   \phi_t= s_t^* \phi_1,
\]
for $t\neq 0$.
\end{proposition}
\begin{proof}
The rescaling map $s_t\colon \g^*\to \g^*$ extends to a Lie algebra morphism
\[ s_t\colon \dd_t\to \dd,\ \xi+\mu\mapsto \xi+t\mu.\]
The inclusions $j_t\colon \g^*\to \dd_t$ defined by 
$j_t(\mu)=\mu+t\,r^\sharp(\mu)$ satisfy 
$s_t\circ j_t=j\circ s_t$. Hence $s_t$ for $t\neq 0$ is an isomorphism of Manin triples with coboundary structure, 
up to rescaling of the metric. In fact, the isomorphism $s_t$ changes the metric by a factor of $t$:
\[ \l s_t(\xi_1+\mu_1),s_t(\xi_2+\mu_2)]\r=
t \l \xi_1+\mu_1,\xi_2+\mu_2\r.\]
Thus, the metric on $\dd_t$ is $\l\cdot,\cdot\r_t=t^{-1}s_t^*\l\cdot,\cdot\r$.  
From our construction, it is immediate that multliplication of the metric on $\dd$ by some scalar 
amounts to multiplication of the form $\sigma$ by the same scalar. This shows 
$\sigma_t=t^{-1} s_t^*\sig$.
Let $A_t=1+\sigma^\flat_t\circ \pi_0^\sharp$ so that 
$\pi_t^\sharp=\pi_0^\sharp\circ A_t^{-1}$. The Lie--Poisson structure satisfies $\pi_0=t\,(s_t^*)\pi_0$, that is, 
$(s_t)_*\circ \pi_0^\sharp\circ s_t^*=t\,\pi_0^\sharp$. 
Using $\sigma_t^\flat=t^{-1}\, s_t^*\circ \sigma_1^\flat\circ (s_t)_*$ this shows that $A_t\circ s_t^*=s_t^*\circ A_1$, 
and consequently $\pi_{t}=t\ s_t^* \pi_{1}$. The scaling behavior of the 2-forms $\sig_t$ implies that the derivative 
$\dot{\sig}_t=\f{\p\sig}{\p t}$ scales as $\dot{\sig}_t=t^{-2} s_t^*\dot{\sig}_1$. 
(Cf.~\cite[Section 2.1]{al:gw}.) This then implies the scaling property of $a_t$. 

In order to show that the family of bisections $\phi_t$ satisfies the property $\phi_t = s_t^* \phi_1$, 
we note that the map $s_t$ lifts to a groupoid automorphism (denoted again by $s_t$) of the
groupoid $T^*G \rra \g^*$. This automorphism can ve viewed 
as the dilation by the factor of $t$ on the fibers of 
the cotangent bundle $T^*G \to G$. In the left trivialization, we have
$s_t: (g, \mu) \mapsto (g, t\mu)$. The bisection $\phi_t$ is defined such that 
the family of maps $F_{\G, t}: (g, \mu) \to (\phi_t(\mu) g, \mu)$ integrates the Moser flow 
of the 1-form $\tilde{a}_t = \tz^* \a_t$ on $\G$.
By Proposition \ref{prop:groupoid} and equation \eqref{eq:omega+sigma},
the symplectic form on $\G \cong T^*G$ is given by $\om_t = \om_0 + \tz^* \sigma_t$
(here $\om_0$ is the canonical symplectic form on $T^*G$).
Since $\om_0=t^{-1} s_t^* \om_0 $ and  $\sigma_t = t^{-1} s_t^* \sigma_1$
we have $\om_t = t^{-1} s_t^* \om_1$.
The Moser 1-form $\tilde{a}_t=\tz^* a_t$ verifies 
$\tilde{a}_t=t^{-2} s_t^* \tilde{a}_1$ which implies the scaling $\tilde{v}_t = t^{-1} s_t^* \tilde{v}_1$ 
of the Moser vector field $\tilde{v}_t$ on $\G$.
And this implies the desired property $\phi_t = s_t^* \phi_1$.

\end{proof}

\section{Functorial properties of linearization}
Suppose $f\colon (M_1,\pi_{M_1})_{x_1}\to (M_2,\pi_{M_2})_{x_2}$ 
is a germ of a Poisson map, where $\pi_{M_1}$ vanishes at $x_1$ and  
$\pi_{M_2}$ vanishes at $x_2$. Then one 
can look for Poisson linearizations of $M_1$ and $M_2$ in which the map $f$ becomes the linear map $T_{x_1}f$. 
In particular, one can consider this problem for morphisms of Poisson Lie groups. 

\subsection{Morphisms of coboundary Poisson Lie groups}
A \emph{morphism of Lie bialgebras} is a Lie algebra morphism 
$\nu\colon \g_1\to \g_2$ preserving cobrackets. It then follows that the dual map
\[ \tau\colon \g_2^*\to \g_1^*\] 
is both a Poisson map and a morphism of Lie bialgebras.  
It exponentiates to a morphism of the simply connected Poisson Lie groups 
\[ \ca{T}\colon G_2^*\to G_1^*.\] 
The following result shows that in the coboundary case, we may choose Poisson linearizations of $G_1^*,G_2^*$ such the map 
$\ca{T}\colon G_2^*\to G_1^*$ becomes the linear map $\tau\colon \g_2^*\to \g_1^*$.
\begin{theorem}\label{th:b}
Suppose $\g_1$ and $\g_2$ are coboundary Lie algebras, and let $\tau\colon \g_2^*\to \g_1^*$ be a morphism of 
the dual Lie bialgebras which integrates to a Poisson Lie group morphism 
$\ca{T}\colon G_2^*\to G_1^*$.
Let $\Exp_i\colon (\g_i^*)_0\to (G_i^*)_e$ be the germs of maps determined by the coboundary structures.  Given a germ of a bisection 
$\psi_1$ such that $\sig_1=\sig_{\psi_1}$, one can choose a germ of a bisection 
$\psi_2$ such that $\sig_2=\sig_{\psi_2}$ and the diagram 
\begin{equation} \label{eq:diagram}
\begin{CD} (\g_2^*)_0 @>>{\tau}> (\g_1^*)_0\\
@V{\on{Exp}_2\circ \A(\psi_2)^{-1}}VV @VV{\on{Exp}_1\circ \A(\psi_1)^{-1}}V \\
{(G_2^*)_e} @>>{\ca{T}}> {(G_1^*)_e}
\end{CD}
\end{equation}
commutes.
\end{theorem}
\begin{proof}
We denote by  $\nu\colon \g_1\to \g_2$ the Lie bialgebra morphism dual to $\tau_0$.
 For all $\xi\in\g_1$, 
\[ e_{\g_2^*}(\nu(\xi))\ \sim_{(\Exp_2,\sig_2)}\ 
e_{G_2^*}(\nu(\xi))\ \sim_{(\ca{T},0)}\  e_{G_1^*}(\xi)
\ \sim_{(\Exp_1,\sig_1)^{-1}}\ e_{\g_1^*}(\xi).\]
Define  $\sig_2'$ by the composition 
\begin{equation}
\label{eq:composition} (\Exp_1,\sig_1)^{-1}\circ (\ca{T},0)\circ (\Exp_2,\sig_2)=(\Exp_1^{-1}\circ \ca{T}\circ \Exp_2,\sig_2').
\end{equation}
That is, $e_{\g_2^*}(\nu(\xi))\sim_{( \Exp_1^{-1}\circ \ca{T}\circ \Exp_2 ,\sig_2')}e_{\g_1^*}(\xi)$, 
which means that $\Exp_1^{-1}\circ \ca{T}\circ \Exp_2$ 
is a moment map relative to the $\sig_2'$-gauge transformed Poisson structure, 
generating the usual coadjoint action of $\g_1$ on $\g_2^*$. Since $\tau$ is a moment map 
for the original Poisson structure, the equivariant Moser method gives a germ of a 
bisection $\psi_2'\in \Gamma(T^*G_2)_0$ relating \eqref{eq:composition} with $(\tau,0)$:
\[
(\Exp_1,\sig_1)^{-1}\circ (\ca{T},0)\circ (\Exp_2,\sig_2)\circ (\A(\psi_2'),\sig_{\psi_2'})^{-1}=(\tau,0).
\]
Let $\psi_2''\in \Gamma(T^*G_2)_0$ be the `pull-back' of the given bisection $\psi_1\in \Gamma(T^*G_1)_0$ under the groupoid morphism lifting $\tau$. 
This bisection satisfies
\[ (\tau,0)\circ (\A(\psi_2''),\sig_{\psi_2''})^{-1}
=(\A(\psi_1),\sig_{\psi_1})^{-1}\circ (\tau,0).\]
Letting $\psi_2=\psi_2''\psi_2'$, we calculate:
\[ \begin{split}\lefteqn{(\Exp_1\circ\, \A(\psi_1)^{-1}\circ \tau,0)}\\
&=(\Exp_1,\sigma_1)\circ (\A(\psi_1),\sig_{\psi_1})^{-1}\circ (\tau,0)\\
&=(\Exp_1,\sigma_1)\circ (\tau,0)\circ (\A(\psi_2''),\sig_{\psi_2''})^{-1}\\
&=(\ca{T},0)\circ (\Exp_2,\sig_2)\circ (\A(\psi_2'),\sigma_{\psi_2'})^{-1}\circ (\A(\psi_2''),\sigma_{\psi_2''})^{-1}\\
&=(\ca{T},0)\circ (\Exp_2,\sig_2)\circ (\A(\psi_2),\sigma_{\psi_2})^{-1}
\end{split}\]
This identity is equivalent to the two equations, 
$\ca{T}\circ \Exp_2\circ\, \A(\psi_2)^{-1}=\Exp_1\circ\,\A(\psi_1)^{-1}\circ \tau$ and 
$\sig_2-\sig_{\psi_2}=0$.
\end{proof}

\begin{remark}
A similar result was obtained in \cite{al:gw} for morphisms of compact Lie groups with the Lu-Weinstein 
Poisson structure. For the group $G=U(n)$, it was used to prove an isomorphism between
the Gelfand-Zeiltin completely integrable system on $\g^*={\rm u}(n)^*$ defined by Guillemin-Sternberg \cite{gu:gc}
and the completely integrable system on $G^*={\rm U}(n)^*$ defined by Flaschka-Ratiu \cite{fl:mo}.
\end{remark}

\section{Addition versus multiplication}
Suppose $G$ is a coboundary Poisson Lie group. 
In this section, we will show that it is possible to choose the Poisson linearizations of $G^*$ and of $G^*\times G^*$ in such a way that the multiplication map becomes the addition in $\g^*$. We will use a technique similar to 
that for the functoriality property (Theorem \ref{th:b}), by comparing two moment maps. 

\subsection{Twisted diagonal action} \label{subsec:twistings}
Let $(\dd,\g,\h)$ be a Manin triple, and $(D,G,H)$ a corresponding triple of Lie groups. 
The group multiplication 
\[ \Mult\colon G^*\times G^*\to G^*\]
of $H=G^*$ is a moment map for the 
\emph{twisted diagonal action} on $G^*\times G^*$. This may be computed as follows:
Let $e_{G^*\times G^*}^{\on{tw}}(\xi)$ be the section of 
$\on{Gr}(\pi_{G^*\times G^*})\subset \T(G^*\times G^*)$
defined by the property 
\[ e_{G^*\times G^*}^{\on{tw}}(\xi)\sim_{(\on{Mult},0)} e_{G^*}(\xi).\]
At any given point $(u_1,u_2)$, this coincides with 
$ e_{G^*\times G^*}^{\on{tw}}(\xi)=(e_{G^*}(\xi_1),e_{G^*}(\xi_2))$ for some $\xi_1,\xi_2\in\g$
(depending on $u_1,u_2$). By considering the pull-back of $\l\theta^R_{G^*},\xi\r$ under $\on{Mult}$, we see that 
$\xi_1=\xi$ and $\xi_2=\Ad^*_{u_1^{-1}}\xi$. 

In the coboundary case, we can change the twisted diagonal action into the usual diagonal action, 
corresponding to the generators $e_{G^*\times G^*}(\xi)=(e_{G^*}(\xi),e_{G^*}(\xi))$:

\begin{proposition}\label{prop:chi}
Suppose $\g$ has a coboundary structure. Then there is a germ of a bisection $\chi$ of the the symplectic groupoid 
$\G\times \G$ of 
$G^*\times G^*$, with the property 
\[ e_{G^*\times G^*}(\xi)\sim_{(\A(\chi),\sigma_\chi)}e^{tw}_{G^*\times G^*}(\xi)\]
for all $\xi\in\g$.
\end{proposition}
A proof of this fact is deferred to the end of this Section.

\subsection{Addition versus multiplication}
We are now in position to prove:
\begin{theorem}\label{th:addmult}
Let $\g$ be a Lie bialgebra with a coboundary structure, and fix 
a germ at $0$ of a bisection $\psi\in\Gamma(T^*G)_0$ with $\sig_\psi=\sig$. 
Then there exists a germ at $0$ of bisection $\phi\in\Gamma(T^*(G\times G))_0$ with 
$\sig_\phi=\pr_1^*\sig+\pr_2^*\sig$, and such that the following diagram commutes:
\[ \begin{CD} (\g^*\times\g^*)_0 @>>{\on{Add}}> (\g^*)_0\\
@V{(\Exp\times\Exp)\circ \A(\phi)^{-1}}VV @VV{\Exp\circ \A(\psi)^{-1}}V \\
{(G^*\times G^*)_e} @>>{\on{Mult}}> {(G^*)_e}
\end{CD}
\]
\end{theorem}
Note that all maps in this diagram are Poisson. 
\begin{proof}
For all $\xi\in\g$, we have
\[ \begin{split}
e_{\g^*\times\g^*}(\xi)&\sim_{(\Exp\times\Exp,\,\pr_1^*\sig+\pr_2^*\sig)}e_{G^*\times G^*}(\xi)
\\&\sim_{(\A(\chi),\sigma_\chi)}
e_{G^*\times G^*}^{tw}(\xi)
\\&\sim_{(\Mult,0)} e_{G^*}(\xi)
\\&\sim_{(\Exp,\sig)^{-1}}e_{\g^*}(\xi).
\end{split}\]
Let $(m,\sigma')$ be the composition of these morphisms:
\[ (m,\sigma')=(\Exp,\sig)^{-1}\circ (\Mult,0)\circ (\A(\chi),\sigma_\chi)\circ (\Exp\times\Exp,\pr_1^*\sig+\pr_2^*\sig).\]
The relation 
$e_{\g^*\times\g^*}(\xi)\sim_{(m,\sigma')}e_{\g^*}(\xi)$
shows that $m$ is a moment map relative to the gauge transformed Poisson structure $\pi_{\g^*\times\g^*}^{\sigma'}$, for the standard (diagonal) coadjoint action of $\g$. On the other hand, we have
\[  e_{\g^*\times\g^*}(\xi)\sim_{(\Add,0)}e_{\g^*}(\xi)\] 
The equivariant Moser method for the bisection $\sig'$ gives a germ of a bisection $\phi'$  with 
\[ (m,\sigma')\circ (\A(\phi'),\sig_{\phi'})^{-1}=(\Add,0).\]
Let $\phi''$ be the `pull-back' of the given bisection $\psi\in \Gamma(T^*G)_0$ under the map $\Add$, thus
\[ (\A(\psi),\sig_\psi)^{-1}\circ (\Add,0)=(\Add,0)\circ (\A(\phi''),\sig_{\phi''})^{-1}.\]
Finally, let $\wt{\chi}$ be the `pull-back' of the bisection $\chi$ under $\Exp\times\Exp$, so that 
\[ (\A(\chi),\sig_\chi)\circ  (\Exp\times\Exp,\pr_1^*\sig+\pr_2^*\sig)=(\Exp\times\Exp,\pr_1^*\sig+\pr_2^*\sig)\circ 
(\A(\wt{\chi}),\sig_{\wt{\chi}}).\]
Put $\phi=\phi''\circ \phi'\circ\wt{\chi}^{-1}$. We obtain
\[ \begin{split}
\lefteqn{(\Exp\circ\, \A(\psi)^{-1}\circ \Add,0)}\\&=(\Exp,\sigma)\circ (\A(\psi),\sigma_\psi)^{-1}\circ  (\Add,0)\\
&=(\Exp,\sigma)\circ  (\Add,0)\circ (\A(\phi''),\sigma_{\phi''})^{-1}\\
&=(\Exp,\sigma)\circ (m,\sigma')\circ (\A(\phi'),\sigma_{\phi'})^{-1}\circ (\A(\phi''),\sigma_{\phi''})^{-1}\\
&=(\Mult,0)\circ (\Exp\times\Exp,\pr_1^*\sigma+\pr_2^*\sigma)\circ (\A(\wt{\chi}),\sigma_{\wt{\chi}})
\circ (\A(\phi'),\sigma_{\phi'})^{-1}\circ (\A(\phi''),\sigma_{\phi''})^{-1}\\
&=(\Mult,0)\circ (\Exp\times\Exp,\pr_1^*\sigma+\pr_2^*\sigma)\circ (\A(\phi),\sigma_\phi)^{-1}\\
&=\Big(\Mult\circ (\Exp\times\Exp)\circ \A(\phi)^{-1},\ (\A(\phi)^{-1})^*(\pr_1^*\sigma+\pr_2^*\sigma-\sigma_\phi)\Big).
\end{split}\]
Hence
$\Mult\circ (\Exp\times\Exp)\circ\,\A(\phi)^{-1}=\Exp\circ \,\A(\psi)^{-1}\circ \Add$
and $\pr_1^*\sigma+\pr_2^*\sigma-\sigma_\phi=0$, as required.
\end{proof}

\begin{remark}
If $G$ is a compact Poisson Lie group with the Lu-Weinstein Poisson structure, the assertions of the Theorem \ref{th:addmult} are valid globally, not only on the level of germs. As mentioned earlier, this is due to the fact that $\Exp$ is a global diffeomorphism in this case, and since the symplectic leaves are compact. A related result was obtained in \cite{al:li}. In more detail, let $\mathcal{O}_1, \mathcal{O}_2 \subset \g^*$ be two coadjoint orbits and let
$$
\mathcal{O}_1+\mathcal{O}_2 = \{ x_1+x_2; \, x_1\in \mathcal{O}_1, x_2\in \mathcal{O}_2\} \subset \g^*.
$$
Since the map $\Exp: \g^* \to G^*$ is equivariant, $\mathcal{D}_1=\Exp(\mathcal{O}_1)$ and $\mathcal{D}_2=\Exp(\mathcal{O}_2)$ are dressing orbits and we define
$$
\mathcal{D}_1\mathcal{D}_2 = \{ u_1u_2; \, u_1\in \mathcal{D}_1, u_2\in \mathcal{D}_2\} \subset G^*.
$$
Theorem \ref{th:addmult} implies that
$$
\Exp(\mathcal{O}_1+\mathcal{O}_2) = \mathcal{D}_1\mathcal{D}_2.
$$
In particular, since $\mathcal{O}_1+\mathcal{O}_2=\mathcal{O}_2+\mathcal{O}_1$,
we obtain an equality $\mathcal{D}_1\mathcal{D}_2=\mathcal{D}_2\mathcal{D}_1$.
\end{remark}

\begin{remark}
The proof of Theorem \ref{th:addmult} is similar to an argument used in the proof from \cite{al:ka} of the Kashiwara-Vergne conjecture for quadratic Lie algebras. In fact, one may use the Theorem to obtain a Kashiwara-Vergne type result for $G^*$, comparing the 
convolution of germs of invariant distributions on $\g^*$ to the convolution of germs of invariant distributions on $G^*$.
\end{remark}

\subsection{Proof of Proposition \ref{prop:chi}}
The left dressing action $\xi\mapsto \xi_{G^*}$ of $\g$ on $G^*$ integrates to a \emph{local} action $\bullet$ of the group $G$.
Dually, we have a local action $\ast$ of $G^*$ on $G$. The two actions are related by the formula 
\[ gu =(g\bullet u)(u^{-1}\ast g)\]
for $u\in G^*$ and $g\in G$ sufficiently close to the group unit.
Using the inclusion $j\colon \g^*\to \dd$, define a germ of a map
\[ \lambda\colon G^*_e\to G_e,\]
by the property $\lambda(u)=u^{-1}\exp(j(\mu))$ for $u=\Exp\mu$. Since the map $j$ is equivariant with respect to the (co)adjoint action, we have 
\begin{equation}\label{eq:lambdaeq}
\lambda(g\bullet u)=(g\bullet u)^{-1}\Ad_g (u\lambda(u))=(u^{-1}\ast g)\lambda(u)g^{-1},\end{equation}
for  $g\in G$ and $u\in G^*$ close to the group unit. 
Identify the (local) symplectic groupoid $\G\rra G^*$ with the (local) action groupoid for the dressing action, and 
let $\chi$ be the germ of bisection of $\G\times\G\rra G^*\times G^*$ given as 
\[ \chi(u_1,u_2)=(e,\lambda(u_1)).\]
Its action on $G^*\times G^*$ is given by $\A(\chi)\colon (u_1,u_2)\mapsto (u_1,\lambda(u_1)\bullet u_2)$. 
\begin{lemma}
The map $\A(\chi)$  intertwines the diagonal action with the twisted diagonal action. 
\end{lemma}
\begin{proof}
The twisted diagonal action of $\g$ on $G^*\times G^*$ integrates to the local group action
$g\bullet(u_1,u_2):=(g\bullet u_1,\ (u_1^{-1}\ast g)\bullet u_2)$, for $g\in G$ and $u_1,u_2\in G^*$ close to the group unit.
Using \eqref{eq:lambdaeq} we find
\[ \begin{split}
\A(\chi)(g\bullet u_1,g\bullet u_2)&=(g\bullet u_1,\ (\lambda(g\bullet u_1)g)\bullet u_2)\\
&=\big(g\bullet u_1,\ ((u_1^{-1}\ast g)\lambda(u_1))\bullet u_2\big)\\
&=g\bullet (u_1,\lambda(u_1)\bullet u_2)\\
&=g\bullet \A(\chi)(u_1,u_2).\qedhere
\end{split}\]
\end{proof}
The Lemma shows that the generating vector fields for the diagonal and twisted diagonal action are $\A(\chi)$-related: 
$\xi_{G^*\times G^*}\sim_{\A(\chi)}\xi^{tw}_{G^*\times G^*}$. Proposition \ref{prop:chi} strengthens this result 
to a relation $e_{G^*\times G^*}(\xi)\sim_{(\A(\chi),\sigma_\chi)} e_{G^*\times G^*}^{tw}(\xi)$ between sections of the cotangent 
Lie algebroid. By Proposition \ref{prop:agcot}, the action of the group of bisections on the cotangent Lie algebroid 
$\on{Gr}(\pi_{G^*\times G^*})\cong T^*(G^*\times G^*)$
coincides with the adjoint action on the Lie algebroid $\A(\G\times \G)$, which in this case is an action Lie algebroid 
for the dressing action. Hence, Proposition \ref{prop:chi} follows from the following result:
\begin{lemma}
The action of $\Ad(\chi)$ on the action Lie algebroid $(G^*\times G^*)\times(\g\times\g)$ takes the constant section 
$(\xi,\xi)$ to the section $(\xi,\Ad^*_{u_1^{-1}}\xi)$.
\end{lemma}
\begin{proof}
We have $\chi^{-1}(u_1,u_2)=(e,\lambda(u_1)^{-1})$. Hence, the adjoint action on the action groupoid takes 
$(u_1,u_2;g_1,g_2)$ to 
\[\begin{split}
\Ad(\chi)(u_1,u_2;g_1,g_2)&=
\big(g_1\bullet u_1,g_2\bullet u_2;e,\lambda(g_1\bullet u_1)\big)\big(u_1,u_2;g_1,g_2\big)
\big(u_1,\lambda(u_1)\bullet u_2;e,\lambda(u_1)^{-1}\big)\\
&=
\big(u_1,\lambda(u_1)\bullet u_2;g_1,\lambda(g_1\bullet u_1)g_2\lambda(u_1)^{-1}\big)
%\\&=\big(u_1,\lambda(u_1)\bullet u_2;g_1,(u_1^{-1}\ast g_1)\lambda(u_1)g_1^{-1}g_2 \lambda(u_1)^{-1}\big).
\end{split}\]
For  $g_1=g_2=g$, this simplifies using the equivariance property of $\lambda$:
\[ \Ad(\chi)(u_1,u_2;g,g)= (u_1,\lambda(u_1)\bullet u_2;g,u_1^{-1}\ast g).\]
The infinitesimal version of this action is as stated in the Lemma. 
\end{proof}

%\bibliographystyle{amsplain}
%\bibliography{../../../Biblio/ref}

\begin{thebibliography}{10}

\bibitem{al:po}
A.~Alekseev, \emph{On {P}oisson actions of compact {L}ie groups on symplectic
  manifolds}, J.~Differential Geom. \textbf{45} (1997), no.~2, 241--256.

\bibitem{al:pur}
A.~Alekseev, H.~Bursztyn, and E.~Meinrenken, \emph{Pure spinors on {L}ie
  groups}, Ast\'{e}risque \textbf{327} (2009), 131--199.

\bibitem{al:ka}
A.~Alekseev and E.~Meinrenken, \emph{Poisson geometry and the
  {K}ashiwara--{V}ergne conjecture}, C.~R.~Math.~Acad.~Sci.~Paris \textbf{335}
  (2002), no.~9, 723--728.

\bibitem{al:gw}
A.~Alekseev and E.~Meinrenken, \emph{Ginzburg-{W}einstein via
  {G}elfand-{Z}eitlin}, J. Differential Geom. \textbf{76} (2007), no.~1, 1--34.

\bibitem{al:li}
A.~Alekseev, E.~Meinrenken, and C.~Woodward, \emph{Linearization of {P}oisson
  actions and singular values of matrix products}, Ann.~Inst.~Fourier
  (Grenoble) \textbf{51} (2001), no.~6, 1691--1717.

\bibitem{bu:ga}
H.~Bursztyn, \emph{On gauge transformations of {P}oisson structures}, Quantum
  field theory and noncommutative geometry, Lecture Notes in Phys., vol. 662,
  Springer, Berlin, 2005, pp.~89--112.

\bibitem{cah:non}
M.~Cahen, S.~Gutt, and J.~Rawnsley, \emph{Nonlinearizability of the {I}wasawa
  {P}oisson {L}ie structure}, Lett. Math. Phys. \textbf{24} (1992), no.~1,
  79--83.

\bibitem{cat:int}
A.~Cattaneo, \emph{On the integration of {P}oisson manifolds, {L}ie algebroids,
  and coisotropic submanifolds}, Lett. Math. Phys. \textbf{67} (2004), no.~1,
  33--48.

\bibitem{cat:int1}
A.~Cattaneo, B.~Dherin, and A.~Weinstein, \emph{Integration of lie algebroid
  comorphisms}, arXiv:1210.4443.

\bibitem{chl:lin}
V.~Chloup-Arnould, \emph{Linearization of some {P}oisson-{L}ie tensor}, J.
  Geom. Phys. \textbf{24} (1997), no.~1, 46--52.

\bibitem{con:noran}
J.F. Conn, \emph{Normal forms for analytic {P}oisson structures}, Ann. of Math.
  (2) \textbf{119} (1984), no.~3, 577--601.

\bibitem{con:norsm}
\bysame, \emph{Normal forms for smooth {P}oisson structures}, Ann. of Math. (2)
  \textbf{121} (1985), no.~3, 565--593.

\bibitem{cos:gro}
A.~Coste, P.~Dazord, and A.~Weinstein, \emph{Groupo\"\i des symplectiques},
  Publications du {D}\'epartement de {M}ath\'ematiques. {N}ouvelle {S}\'erie.
  {A}, {V}ol.\ 2, Publ. D\'ep. Math. Nouvelle S\'er. A, vol.~87, Univ.
  Claude-Bernard, Lyon, 1987, pp.~i--ii, 1--62. \MR{996653 (90g:58033)}

\bibitem{cou:di}
T.~Courant, \emph{Dirac manifolds}, Trans.~Amer.~Math.~Soc. \textbf{319}
  (1990), no.~2, 631--661.

\bibitem{cra:geo}
M.~Crainic and R.L. Fernandes, \emph{A geometric approach to {C}onn's
  linearization theorem}, Ann. of Math. (2) \textbf{173} (2011), no.~2,
  1121--1139.

\bibitem{cra:exi}
M.~Crainic and I.~Marcut, \emph{On the existence of symplectic realizations},
  J. Symplectic Geom. \textbf{9} (2011), no.~4, 435--444.

\bibitem{sme:exi}
V.~de~Smedt, \emph{Existence of a {L}ie bialgebra structure on every {L}ie
  algebra}, Lett. Math. Phys. \textbf{31} (1994), no.~3, 225--231.

\bibitem{dr:ha}
V.~G. Drinfel{\cprime}d, \emph{Hamiltonian structures on {L}ie groups, {L}ie
  bialgebras and the geometric meaning of classical {Y}ang-{B}axter equations},
  Dokl.~Akad.~Nauk SSSR \textbf{268} (1983), no.~2, 285--287.

\bibitem{dr:qu}
V.~G. Drinfeld, \emph{Quantum groups}, Proceedings of the International
  Congress of Mathematicians, Vol.~1, 2 (Berkeley, Calif., 1986) (Providence,
  RI), Amer.~Math.~Soc., 1987, pp.~798--820.

\bibitem{dr:quas}
V.~G. Drinfel{\cprime}d, \emph{Quasi-{H}opf algebras}, Algebra i Analiz
  \textbf{1} (1989), no.~6, 114--148. \MR{1047964 (91b:17016)}

\bibitem{duf:lin}
J.-P. Dufour, \emph{Lin\'earisation de certaines structures de {P}oisson}, J.
  Differential Geom. \textbf{32} (1990), no.~2, 415--428.

\bibitem{enr:co}
B.~Enriquez, P.~Etingof, and I.~Marshall, \emph{Comparison of {P}oisson
  structures and {P}oisson-{L}ie dynamical {$r$}-matrices}, Int. Math. Res.
  Not. (2005), no.~36, 2183--2198.

\bibitem{fer:lin}
R.~Fernandes and P.~Monnier, \emph{Linearization of {P}oisson brackets}, Lett.
  Math. Phys. \textbf{69} (2004), 89--114.
  
\bibitem{fl:mo}
H.~Flaschka and T.~Ratiu, \emph{A {M}orse-theoretic proof of {P}oisson {L}ie
  convexity}, Integrable systems and foliations/Feuilletages et syst\`emes
  int\'egrables (Montpellier, 1995), Progr.~Math., vol. 145, Birkh\"auser
  Boston, Boston, MA, 1997, pp.~49--71.  

\bibitem{gi:lp}
V.~L. Ginzburg and A.~Weinstein, \emph{Lie-{P}oisson structure on some
  {P}oisson {L}ie groups}, J.~Amer.~Math.~Soc. \textbf{5} (1992), no.~2,
  445--453.

\bibitem{gu:gc}
V.~Guillemin and S.~Sternberg, \emph{The {G}elfand-{C}etlin system and
  quantization of the complex flag manifolds}, J.~Funct.~Anal \textbf{52}
  (1983), 106--128. 

\bibitem{hig:dua}
P.~Higgins and K.~Mackenzie, \emph{Duality for base-changing morphisms of
  vector bundles, modules, {L}ie algebroids and {P}oisson structures}, Math.
  Proc. Cambridge Philos. Soc. \textbf{114} (1993), no.~3, 471--488.

\bibitem{kar:ana}
M.~V. Karas{\"e}v, \emph{Analogues of objects of the theory of {L}ie groups for
  nonlinear {P}oisson brackets}, Izv. Akad. Nauk SSSR Ser. Mat. \textbf{50}
  (1986), no.~3, 508--538, 638.

\bibitem{lu:mo}
J.-H. Lu, \emph{Momentum mappings and reduction of {P}oisson actions},
  Symplectic geometry, groupoids, and integrable systems (Berkeley, CA, 1989),
  Springer, New York, 1991, pp.~209--226.

\bibitem{lu:gr}
J.-H. Lu and A.~Weinstein, \emph{Groupo\"\i des symplectiques doubles des
  groupes de {L}ie-{P}oisson}, C. R. Acad. Sci. Paris S\'er. I Math.
  \textbf{309} (1989), no.~18, 951--954.

\bibitem{maj:fou}
S.~Majid, \emph{Foundations of quantum group theory}, Cambridge University
  Press, Cambridge, 1995.

\bibitem{me:clifbook}
E.~Meinrenken, \emph{{L}ie groups and {C}lifford algebras}, {Ergebnisse der
  Mathematik und ihrer Grenzgebiete}, vol.~58, Springer, Heidelberg, 2013.

\bibitem{pra:tro}
Jean Pradines, \emph{Troisi\`eme th\'eor\`eme de {L}ie les groupo\"\i des
  diff\'erentiables}, C. R. Acad. Sci. Paris S\'er. A-B \textbf{267} (1968),
  A21--A23.

\bibitem{wad:nor}
A.~Wade, \emph{Normalisation formelle de structures de {P}oisson}, C. R. Acad.
  Sci. Paris S\'er. I Math. \textbf{324} (1997), no.~5, 531--536.

\bibitem{wei:loc}
A.~Weinstein, \emph{The local structure of {P}oisson manifolds}, J.
  Differential Geom. \textbf{18} (1983), no.~3, 523--557.

\bibitem{wei:pri}
A.~Weinstein, \emph{Poisson geometry of the principal series and
  nonlinearizable structures}, J. Differential Geom. \textbf{25} (1987), no.~1,
  55--73.

\bibitem{wei:sym}
A.~Weinstein, \emph{Symplectic groupoids and {P}oisson manifolds}, Bull. Amer.
  Math. Soc. (N.S.) \textbf{16} (1987), no.~1, 101--104.

\bibitem{xu:poigr}
Ping Xu, \emph{On {P}oisson groupoids}, Internat. J. Math. \textbf{6} (1995),
  no.~1, 101--124.

\end{thebibliography}

\def\cprime{$'$} \def\polhk#1{\setbox0=\hbox{#1}{\ooalign{\hidewidth
  \lower1.5ex\hbox{`}\hidewidth\crcr\unhbox0}}} \def\cprime{$'$}
  \def\cprime{$'$} \def\cprime{$'$}
  \def\polhk#1{\setbox0=\hbox{#1}{\ooalign{\hidewidth
  \lower1.5ex\hbox{`}\hidewidth\crcr\unhbox0}}} \def\cprime{$'$}
  \def\cprime{$'$} \def\cprime{$'$} \def\cprime{$'$} \def\cprime{$'$}
\providecommand{\bysame}{\leavevmode\hbox to3em{\hrulefill}\thinspace}
\providecommand{\MR}{\relax\ifhmode\unskip\space\fi MR }
\providecommand{\MRhref}[2]{%
  \href{http://www.ams.org/mathscinet-getitem?mr=#1}{#2}
}
\providecommand{\href}[2]{#2}

\end{document}